\RequirePackage[l2tabu, orthodox]{nag}
\documentclass[11pt]{article}
\usepackage{amssymb, amsmath, amsthm}
\usepackage{mathtools}
\usepackage[all]{xy}
\usepackage{bm}
\usepackage{calc}
\usepackage{mathrsfs}
\usepackage{lmodern}
\usepackage{microtype}
\usepackage[hmargin = 3cm, vmargin = 2cm, includefoot]{geometry}
\usepackage{xparse}

\usepackage[colorlinks = true, pdfstartview = FitV, linkcolor = blue, citecolor = blue, urlcolor = blue]{hyperref}
\usepackage[alphabetic, lite]{amsrefs}

\numberwithin{equation}{section}

\DeclareMathOperator{\supp}{supp}
\DeclareMathOperator{\id}{id}

\DeclareMathOperator{\muSupp}{\mu supp}
\newcommand*{\unpmuSupp}{\unp{\muSupp}}
\DeclareMathOperator{\Ppg}{Ppg}
\DeclareMathOperator{\codim}{codim}

\DeclareMathOperator{\Lip}{Lip}
\newcommand*{\Lippw}{\Lip^{\mathrm{pw}}}

\renewcommand*{\leq}{\leqslant}
\renewcommand*{\geq}{\geqslant}

\makeatletter
\newcommand{\subalign}[1]{%
  \vcenter{%
    \Let@ \restore@math@cr \default@tag
    \baselineskip\fontdimen10 \scriptfont\tw@
    \advance\baselineskip\fontdimen12 \scriptfont\tw@
    \lineskip\thr@@\fontdimen8 \scriptfont\thr@@
    \lineskiplimit\lineskip
    \ialign{\hfil$\m@th\scriptstyle##$&$\m@th\scriptstyle{}##$\crcr
      #1\crcr
    }%
  }
}
\makeatother

\newcommand*{\fantome}[3][]{\mathmakebox[\maxof{\width}{\widthof{$#1$}}][#2]{#3}}

\ExplSyntaxOn
\DeclareExpandableDocumentCommand{\IfNoValueOrEmptyTF}{mmm}{\IfNoValueTF{#1}{#2}{\tl_if_empty:nTF {#1}{#2}{#3}}}
\ExplSyntaxOff

\NewDocumentCommand\Lagr{soom}
{
\IfBooleanTF{#1}
{
\IfNoValueTF{#3}
  {\IfNoValueTF{#2}{{\Lambda_{#4}}}{\Lambda_{#2}\left({#4}\right)}}
  {\Lambda_{#2}^{#3}\left({#4}\right)}
}
{
\IfNoValueTF{#3}
  {\IfNoValueTF{#2}{{\Lambda_{#4}}}{\Lambda_{#2}({#4})}}
  {\Lambda_{#2}^{#3}({#4})}
}
}

\NewDocumentCommand\unpLagr{soom}
{
\IfBooleanTF{#1}
{
\IfNoValueTF{#3}
  {\IfNoValueTF{#2}{{\unp{\Lambda}_{#4}}}{\unp{\Lambda}_{#2}\left({#4}\right)}}
  {\unp{\Lambda}_{#2}^{#3}\left({#4}\right)}
}
{
\IfNoValueTF{#3}
  {\IfNoValueTF{#2}{{\unp{\Lambda}_{#4}}}{\unp{\Lambda}_{#2}({#4})}}
  {\unp{\Lambda}_{#2}^{#3}({#4})}
}
}

\NewDocumentCommand\Whit{soom}
{
\IfBooleanTF{#1}
{
\IfNoValueTF{#3}
  {\IfNoValueTF{#2}{{C_{#4}}}{C_{#2}\left({#2}\right)}}
  {C_{#2}^{#3}\left({#4}\right)}
}
{
\IfNoValueTF{#3}
  {\IfNoValueTF{#2}{{C_{#4}}}{C_{#2}({#4})}}
  {C_{#2}^{#3}({#4})}
}
}

\NewDocumentCommand\unpWhit{soom}
{
\IfBooleanTF{#1}
{
\IfNoValueTF{#3}
  {\IfNoValueTF{#2}{{\unp{C}_{#4}}}{\unp{C}_{#2}\left({#2}\right)}}
  {\unp{C}_{#2}^{#3}\left({#4}\right)}
}
{
\IfNoValueTF{#3}
  {\IfNoValueTF{#2}{{\unp{C}_{#4}}}{\unp{C}_{#2}({#4})}}
  {\unp{C}_{#2}^{#3}({#4})}
}
}

\DeclarePairedDelimiter{\intervO}{]}{[}
\DeclarePairedDelimiter{\intervCO}{[}{[}
\DeclarePairedDelimiter{\intervOC}{]}{]}
\DeclarePairedDelimiter{\intervC}{[}{]}
\DeclarePairedDelimiterX{\innerp}[2]{\langle}{\rangle}{#1,#2}
\DeclarePairedDelimiter{\abs}{\lvert}{\rvert}
\DeclarePairedDelimiter{\norm}{\lVert}{\rVert}

\newcommand*{\fcomp}{\circ}

\newcommand*{\comp}[1][]{\underset{#1}{\circ}}
\newcommand*{\acomp}[1][]{\underset{#1}{\overset{a}{\circ}}}

\newcommand*{\conv}[1][]{\underset{#1}{\circ}}

\newcommand*{\V}{\mathbb{V}}
\newcommand*{\W}{\mathbb{W}}
\newcommand*{\N}{\mathbb{N}}
\newcommand*{\Z}{\mathbb{Z}}
\newcommand*{\Q}{\mathbb{Q}}
\newcommand*{\R}{\mathbb{R}}

\newcommand*{\cor}{\Bbbk}

\newcommand*{\ovD}{\overline{D}}
\newcommand*{\ovd}{\overline{d}}
\newcommand*{\unD}{\underline{D}}
\newcommand*{\und}{\underline{d}}

\newcommand*{\ovQ}{\overline{Q}}
\newcommand*{\unQ}{\underline{Q}}

\newcommand*{\RD}{\operatorname{D}}

\newcommand*{\define}[1]{\textbf{#1}}

\newcommand*{\Derb}{\mathrm{D}^{\mathrm{b}}}
\newcommand*{\rsect}{\mathrm{R}\Gamma}

\newcommand*{\unp}{\bm{\dot}}
\newcommand*{\Ltens}{\otimes^\mathrm{L}}

\newcommand*{\RHom}[1][]{\mathrm{R}\mathrm{Hom}_{\raise1.5ex\hbox to.1em{}#1}}
\renewcommand*{\hom}[1][]{{\mathscr{H}\mspace{-4mu}om}_{\raise1.5ex\hbox to.1em{}#1}}
\newcommand*{\rhom}[1][]{{\mathrm{R}\mathscr{H}\mspace{-3mu}om}_{\raise1.5ex\hbox to.1em{}#1}}
\DeclareMathOperator{\muhom}{\mu\mspace{-1mu}hom}

\newcommand*{\oim}[1]{{#1}_*}
\newcommand*{\eim}[1]{{#1}_{\, !}}
\newcommand*{\roim}[1]{\mathrm{R}{#1}_*}
\newcommand*{\reim}[1]{\mathrm{R}{#1}_{\, !}}
\newcommand*{\opb}[1]{#1^{-1}}
\newcommand*{\epb}[1]{#1^{\, !} \,}

\renewcommand*{\to}[1][]{\xrightarrow[]{#1}}

\newcommand*{\isoto}[1][]{\xrightarrow[#1]{{\raisebox{-.6ex}[0ex][-.6ex]{$\mspace{1mu}\sim\mspace{2mu}$}}}}

\newcommand*{\closure}[1]{\overline{#1}}
\newcommand*{\fwclosure}[1]{\operatorname{cl_{fw}}(#1)}

\newcommand*{\interior}[1]{\operatorname{Int}\left(#1\right)}
\newcommand*{\interiorSmall}[1]{\operatorname{Int}(#1)}

\newcommand*{\indlim}[1][]{\mathop{\varinjlim}\limits_{#1}}

\newcommand*{\spa}{\vspace{1.2ex}\noindent}

\theoremstyle{plain}
\newtheorem{theorem}{Theorem}[section]
\newtheorem{proposition}[theorem]{Proposition}

\newtheorem{lemma}[theorem]{Lemma}
\newtheorem{corollary}[theorem]{Corollary}
\theoremstyle{definition}
\newtheorem{definition}[theorem]{Definition}

\newtheorem{example}[theorem]{Example}
\newtheorem{remark}[theorem]{Remark}
\newtheorem*{erratum}{Erratum}

\title{A microlocal characterization of Lipschitz continuity}
\author{Beno\^it Jubin%
}
\date{\today}

\begin{document}

\maketitle

\begin{abstract}

We study continuous maps between differential manifolds from a microlocal point of view.
In particular, we characterize the Lipschitz continuity of these maps in terms of the microsupport of the constant sheaf on their graph.
Furthermore, we give lower and upper bounds on the microsupport of the graph of a continuous map and use these bounds to characterize strict differentiability in microlocal terms.
\end{abstract}

\small
\tableofcontents
\normalsize

\section*{Introduction}
\label{sec:introduction}
\addcontentsline{toc}{section}{\nameref{sec:introduction}}

Microlocal analysis is the study of phenomena occurring on differential manifolds via a study in their cotangent bundle; for instance, the study of the singularities of solutions of a partial differential equation on a manifold $M$ via the study of their wavefront set in $T^*M$.
A general setting for microlocal analysis is the microlocal theory of sheaves, developed by M.~Kashiwara and P.~Schapira (see~\cite{KS}).
In~\cite{Vic}, N.~Vichery used this theory to study from a microlocal viewpoint continuous real-valued functions on differential manifolds, and to define for these functions a good notion of subdifferential.
We extend this study to continuous maps between differential manifolds.
We study simultaneously the tangent aspects of the subject to emphasize the parallelism between the tangent and cotangent sides.

Specifically, let $f \colon M \to N$ be a continuous map between differential manifolds.
We denote its graph by $\Gamma_f \subseteq M \times N$.
We define its \define{Whitney cone} $\Whit{f}$ as the Whitney cone of its graph and its \define{conormal} $\Lagr{f}$ as the microsupport of the constant sheaf on its graph, that is,
\begin{align}
\Whit{f} &\coloneqq C(\Gamma_f, \Gamma_f) \subseteq T(M \times N)
\qquad\text{and}\\
\Lagr{f} &\coloneqq \fantome[C(\Gamma_f, \Gamma_f)]{l}{\muSupp(\cor_{\Gamma_f})} \subseteq T^*(M \times N)
\end{align}
where $\cor$ is any nonzero commutative ring of finite global dimension (for instance $\Z$ or a field).
All these terms and pieces of notation are precisely defined in the article.

The Whitney cone $\Whit{f}$ is a closed symmetric cone and the conormal $\Lagr{f}$ is a coisotropic closed symmetric cone.
If $f$ is $C^1$, then its Whitney cone is equal to the tangent bundle of its graph and its conormal is equal to the conormal bundle of its graph, that is, $\Whit{f} = T\Gamma_f$ and $\Lagr{f} = (T\Gamma_f)^\perp$.

We prove that $f$ is Lipschitz if and only if its Whitney cone contains no nonzero ``vertical'' vectors, that is, $\Whit{f} \cap (0_M \times TN) \subseteq 0_{MN}$, if and only if its conormal contains no nonzero ``horizontal'' covectors, that is, $\Lagr{f} \cap (T^*M \times 0^*_N) \subseteq 0^*_{MN}$.

To prove these results, we use the microlocal theory of sheaves of Kashiwara and Schapira, which we review in Section~\ref{sec:background}.

In Section~\ref{sec:subsets}, we review the main properties of microsupports associated with subsets.
If $A$ is a locally closed subset of $M$, we set $\muSupp(A) \coloneqq \muSupp(\cor_A)$ and denote its tangent cone by $C(A)$ and its strict tangent cone by $N(A)$.
If $C$ is a cone in $TM$, we denote its polar by $C^\circ$.
We give a direct proof of the following known bounds: if $Z \subseteq M$ is closed, then $\opb{\pi_M}(Z) \cap C(Z)^\circ \subseteq \muSupp(Z) \subseteq N(Z)^\circ$.

In Section~\ref{sec:whitney}, we define the Whitney cone of a continuous map and give its first properties.
In particular, we characterize Lipschitz continuity and strict differentiability in terms of the Whitney cone and extend these characterizations to topological submanifolds.
We also prove the following chain rule, which will be needed later.
Let $f_i \colon M_i \to M_{i+1}$ be continuous maps between differential manifolds for $i \in \{ 1, 2 \}$.
If $(0_1 \comp \Whit{f_1}) \cap (\Whit{f_2} \comp 0_3) \subseteq 0_2$, for instance if $f_1$ is Lipschitz, then
\begin{equation}
\Whit{f_2 \fcomp f_1} \subseteq \Whit{f_1} \comp \Whit{f_2}
\end{equation}
with equality if $f_2$ is $C^1$.

In Section~\ref{sec:conormal}, we define the conormal of a continuous map and give its first properties.
We use the convolution of kernels to extend to continuous maps the functorial properties of the microsupport for the four image operations.
We also use it to prove the following chain rule.
With the notation above, if $(0^*_1 \comp \Lagr{f_1}) \cap (\Lagr{f_2} \comp 0^*_3) \subseteq 0^*_2$, for instance if $f_2$ is Lipschitz, then
\begin{equation}
\Lagr{f_2 \fcomp f_1} \subseteq \Lagr{f_1} \acomp \Lagr{g_2}
\end{equation}
with equality if $f_1$ is a $C^1$-submersion (and if $f_1$ and $f_2$ are both $C^1$).

In Section~\ref{sec:real-valued}, we study the case of real-valued functions.
We define directional Dini derivatives in order to describe more precisely the various cones associated with a function.
We also study local extrema and more generally ``extrema at first order'', which we define in that section.
Namely, if $x \in M$, we set $\Whit[x]{f} \coloneqq (\Whit{f})_{(x, f(x))}$ and $\Lagr[x]{f} \coloneqq (\Lagr{f})_{(x, f(x))}$.
Then, we prove the following generalization of Fermat's lemma to continuous functions: if $f \colon M \to \R$ is continuous and has a first-order extremum at~$x \in M$, then
\begin{align}
T_x M \times \{0\} &\subseteq \Whit[x]{f}
\qquad\text{and}\\
\fantome[T_x M]{c}{\{0\}} \times \fantome[\{0\}]{c}{\R} &\subseteq \Lagr[x]{f}.
\end{align}
Finally, we relate the conormal of a function to the microsupport of the constant sheaf on its epigraph, studied by N.~Vichery in~\cite{Vic}, and we prove that the two points of view are equivalent for Lipschitz functions.

Section~\ref{sec:general} is the main section of the paper, where the claimed characterizations of Lipschitz continuity and strict differentiability are proved.
First, we prove analogues of Rolle's lemma and the mean value theorem for continuous maps between vector spaces.
This allows to give the following upper bound on the Whitney cone of a continuous map in terms of its conormal.
To state it, we define the following analogues of the directional derivatives.
If $(x, u) \in TM$ and $(x, \eta) \in M \times_N T^*N$, we set
\begin{align}
\Whit[x][u]{f} &\coloneqq \Whit[x]{f} \cap \big( \R_{\geq 0} u \times \fantome[\R_{\geq 0} \eta]{c}{T_{f(x)} N} \big)
\qquad\text{and}\\
\Lagr[x][\eta]{f} &\coloneqq \Lagr[x]{f} \cap \big( \fantome[\R_{\geq 0} v]{c}{T^*_xM} \times \fantome[T_{f(x)} N]{c}{\R_{\geq 0} \eta} \big).
\end{align}
For a subset $A$ of a vector space, we set $\unp{A} \coloneqq A \setminus \{0\}$ and $A^\top \coloneqq \bigcup_{v \in \unp{A}} v^\perp$.
Then,
\begin{equation}
\Whit[x]{f} \subseteq \bigcap_{\eta \in \unp{T}^*_{f(x)}M} \Lagr[x][\eta]{f}^\top
\end{equation}
with equality if $\dim N = 1$, in which case $\Whit{f} = \Lagr{f}^\top$.
We use this bound to prove the microlocal characterization of Lipschitz continuity.
This allows us to prove the following upper bound:
\begin{equation}
\Lagr[x]{f} \subseteq \bigcap_{u \in \unp{T}_xM} \Whit[x][u]{f}^\top
\end{equation}
with equality if $\dim M = 1$, in which case $\Lagr{f} = \Whit{f}^\top$.
This in turn allows us to characterize strict differentiability in terms of the Whitney cone and of the conormal.
We give applications of these results to the theory of causal manifolds.

In Section~\ref{sec:submanifolds}, we generalize some of these results to topological submanifolds, and we give conditions in terms of the Whitney cone and of the conormal in order that such submanifolds be locally graphs of Lipschitz or strictly differentiable maps.
For instance, if $M$ is a closed topological submanifold of a differential manifold, then
\begin{equation}
\muSupp(M) \subseteq C(M, M)^\top
\end{equation}
with equality if $\dim M = 1$.

Some results (but not all) also hold if one replaces ``Lipschitz'' (resp.\ ``strictly differentiable'', $\Lip$, $\Whit{f}$, $\Lagr{f}$) with ``pointwise Lipschitz'' (resp.\ ``differentiable'', $\Lippw$, $C(\Gamma_f)$, $\complement \Ppg(f)$), but we do not state them.

\begin{erratum}
As explained in Remark~\ref{remark:erratum}, Proposition~1.12 of~\cite{JS} is misstated.
We give in that remark the correct statement and explain why this has no consequences on the rest of~\cite{JS}.
\end{erratum}

\paragraph{Acknowledgments}
I would like to thank Pierre Schapira for many fruitful discussions.

\section{Background material}
\label{sec:background}

\subsection{Notation and conventions}
\label{subsec:notation}

Unless otherwise specified,
\begin{itemize}
\item
the symbol $\cor$ denotes a nonzero commutative ring of finite global dimension (for instance $\Z$ or a field),
\item
vector spaces and manifolds are real and finite-dimensional,
\item
manifolds are paracompact Hausdorff,
\item
manifolds, morphisms of manifolds, and submanifolds, are smooth, that is, of class $C^\infty$, and submanifolds are embedded (hence locally closed),
\item
topological (sub)manifolds are called $C^0$-(sub)manifolds, and $C^0$-submanifolds are locally flatly embedded (that is, their inclusion is locally $C^0$-isomorphic to a linear inclusion $\R^m \hookrightarrow \R^n$) hence locally closed.
\end{itemize}

We use the terms ``function'' and ``map'' interchangeably.

\paragraph{Sets}
Given some sets $X_i$, we set for short $X_{ij} \coloneqq X_i \times X_j$ and similarly for $X_{ijk}$, and we write $p_i$ (resp.\ $p_{ij}$) for any projection from a product of the $X_j$'s (which will be clear from the context) to $X_i$ (resp.\ to $X_{ij}$).
For a product of the form $X \times Y$, we also denote the projections by $p_X \colon X \times Y \to X$ and $p_Y \colon X \times Y \to Y$, and use the same notation for a pullback $X \times_Z Y$.
The diagonal map of $X$ is denoted by $\delta_X \colon X \to X \times X$, and the diagonal of $X$ by $\Delta_X \coloneqq \delta_X(X)$, or simply by $\delta$ and $\Delta$ if there is no risk of confusion.

If $R_i \subseteq X_i \times X_{i+1}$ for $i \in \{1, 2\}$ are relations, we define the \define{composite relation}
\begin{equation}
R_1 \comp R_2 \coloneqq p_{13} \big( \opb{p_{12}}(R_1) \cap \opb{p_{23}}(R_2) \big).
\end{equation}
If $a$ is an involution of $X_2$, we set $p_{12^a} \coloneqq (\id_{X_1} \times a) \comp p_{12}$, and similarly for other indices, and we set
\begin{equation}
R_1 \acomp R_2 \coloneqq p_{13} \big( \opb{p_{12^a}}(R_1) \cap \opb{p_{23}}(R_2) \big)
= p_{13} \big( \opb{p_{12}}(R_1) \cap \opb{p_{2^a3}}(R_2) \big).
\end{equation}
If there is a risk of confusion, we write the composition as $\comp[2]$ and the twisted composition as $\acomp[2]$.

If $A_i \subseteq X_i$ for $i \in \{1, 2\}$ and $R \subseteq X_1 \times X_2$, then we define $A_1 \comp R \coloneqq p_2 \big( \opb{p_1}(A_1) \cap R \big)$ and $R \comp A_2 \coloneqq p_1 \big( R \cap \opb{p_2}(A_2) \big)$, and if $a$ is an involution of $X_1$ or $X_2$ respectively, $A_1 \acomp R \coloneqq p_2 \big( \opb{p_{1^a}}(A_1) \cap R \big)$ and $R \acomp A_2 \coloneqq p_1 \big( R \cap \opb{p_{2^a}}(A_2) \big)$.
These definitions can be considered as special cases of the previous paragraph, by identifying for instance $A_1$ with the relation $R_0 = \{ (\varnothing, x) \in X_0 \times X_1 \mid x \in A_1 \}$ with $X_0 = \{\varnothing\}$.

If $X$ and $Y$ are two sets and $f \colon X \to Y$ is a function, we denote by $\Gamma_f \subseteq X \times Y$ its \define{graph}.
We will often use implicitly the isomorphism $p_X|_{\Gamma_f} \colon \Gamma_f \isoto X$, with inverse $(\id_X, f)$.
This is also an isomorphism of manifolds if $f$ is a morphism of manifolds.
If $f_i \colon X_i \to X_{i+1}$ for $i \in \{1, 2 \}$, then $\Gamma_{f_{i+1} \fcomp f_i} = \Gamma_{f_i} \comp \Gamma_{f_{i+1}}$ (note the usual backward composition).
If $f \colon X_1 \to X_2$ and $A_i \subseteq X_i$ for $i \in \{1, 2\}$, then $f(A_1) = A_1 \comp \Gamma_f$ and $\opb{f}(A_2) = \Gamma_f \comp A_2$.
If $f \colon X \to \R$, we set $\{ f \leq 0 \} \coloneqq \{ x \in X \mid f(x) \leq 0 \}$, and similarly for ``$< 0$'', etc.

Given a real-valued function $f \colon X \to \R$, we denote its \define{epigraph} and \define{hypograph} by
\begin{equation}
\Gamma^\pm_f \coloneqq \{ (x, t) \in X \times \R \mid \pm (t - f(x)) \geq 0 \}.
\end{equation}

\paragraph{Topological spaces}
Given any subset $A$ of a topological space, we denote by $\closure{A}$ its closure, by $\interior{A}$ its interior, and by $\partial A \coloneqq \closure{A} \setminus \interior{A}$ its boundary.

A \define{topological embedding} is a continuous map that is an isomorphism onto its image.
A \define{topological immersion} is a map that is locally a topological embedding.
A continuous map is \define{proper} if it is universally closed (that is, all its pullbacks are closed) or equivalently if it is closed with compact fibers.

Let $R_i \subseteq X_i \times X_{i+1}$ for $i \in \{ 1, 2 \}$ be relations on topological spaces.
If $p_{13}$ is proper on $\closure{\opb{p_{12}}(R_1) \cap \opb{p_{23}}(R_2)}$, then $\closure{R_1 \comp R_2} \subseteq \closure{R_1} \comp \closure{R_2}$.
In particular, if $R_1$ and $R_2$ are closed, then under the above condition, $R_1 \comp R_2$ is closed.

Given an extended real-valued function $f \colon X \to \closure{\R}$, we define the function ${\liminf f \colon}\allowbreak X \to \closure{\R}, x \mapsto \liminf_{y \to x} f(y)$.
It is characterized by $\Gamma^+_{\liminf f} = \closure{\Gamma^+_f}$.
An extended real-valued function is lower-semicontinuous if and only if its epigraph is closed, if and only if it is equal to its $\liminf$.
We define similarly $\limsup f$, characterized by $\Gamma^-_{\limsup f} = \closure{\Gamma^-_f}$.

\paragraph{Vector spaces}
Let $\V$ be a vector space and let $A \subseteq \V$.
We set $A^{a} \coloneqq - A$ and $\unp{A} \coloneqq A \setminus \{0\}$.
The subset $A$ is \define{conic} (or is a \define{cone}) if $\R_{> 0}A = A$ and is \define{symmetric} if $A = A^a$.
Note that a nonempty symmetric convex cone is a vector subspace.

We denote respectively the \define{orthogonal} and the \define{polar} of $A$ by
\begin{align}
\quad\qquad
A^\perp &\coloneqq \{ \xi \in \V^* \mid \forall v \in A, \innerp{\xi}{v} = 0 \}
\qquad\text{and}\\
\fantome[A^\perp]{l}{A^\circ} &\coloneqq \{ \xi \in \V^* \mid \forall v \in A, \innerp{\xi}{v} \geq 0 \}
\end{align}
and we define
\begin{equation}
A^\top \coloneqq \{ \xi \in \V^* \mid \exists v \in \unp{A}, \innerp{\xi}{v} = 0 \}.
\end{equation}
Setting $I_\V \coloneqq \{ (v, \xi) \in \V \times \V^* \mid \innerp{\xi}{v} = 0\}$, one has $A^\top = \unp{A} \comp I_\V$.
Note that $A^\top = \bigcup_{v \in \unp{A}} v^\perp$, so $(-)^\top$ is increasing.
One has $\V^* \setminus A^\top = \{ \xi \in \V^* \mid \xi^\perp \cap A \subseteq \{0\} \}$.
If $A$ is compact or is a closed cone, then $A^\top$ is closed.

\paragraph{Vector bundles}
Let $p \colon E \to B$ be a vector bundle.
One denotes by $a_E \colon E \to E$ the \define{antipodal map} (that is, the fiberwise opposite) and by $0_E$ the \define{zero section} (or its image in $E$).
If $A \subseteq E$ and $x \in B$, we set $A_x \coloneqq A \cap \opb{p}(x)$.
A subset $A$ of $E$ is defined to be a cone (resp.\ to be symmetric, convex, a vector subspace) if all the $A_x$'s are.
Note however that $A$ being closed (resp.\ open) implies that all the $A_x$'s are, but the converse is false in general.
Similarly, the vector space operations, the polar, orthogonal, removal of the origin, and the operation $(-)^\top$ are done fiberwise (but not the operations of closure, interior, and boundary).
The polar of an open subset of a vector bundle is a closed convex cone (see for instance~\cite{JS}*{Lemma~1.2}).

\paragraph{Manifolds}
Let $M$ be a manifold.
We denote by $\tau_M \colon TM \to M$ the tangent bundle of $M$ and by $\pi_M \colon T^*M \to M$ its cotangent bundle, and simply write $\tau$ and $\pi$ if there is no risk of confusion.
For short, we denote the antipodal maps $a_{TM}$ and $a_{T^*M}$ by $a_M$ and the zero sections $0_{TM}$ and $0_{T^*M}$ by $0_M$ and $0^*_M$ respectively.
For various projections, we may write $p_1$ instead of, for instance, $p_{TM_1}$, etc.

For a submanifold $N$ of $M$, we denote by $T_N M \coloneqq (N \times_M TM)/TN \to N$ its \define{normal bundle} and by $T^*_N M \to N$ its \define{conormal bundle} (the subbundle of $N \times_M T^*M \to N$ orthogonal to $TN \to N$).

Let $f \colon M \to N$ be a morphism of manifolds.
One has the following commutative diagram, with the obvious maps.
\begin{equation}
\xymatrix{
TM \ar[dr]_{\tau_M} \ar[r]^-{f'} \ar@/^2.2pc/[rr]^{Tf} & M \times_N TN \ar[d]^{p_M} \ar[r]^-{f_\tau} & TN \ar[d]^{\tau_N}\\
& M \ar[r]^f & N
}
\end{equation}
Note that $T\Gamma_f = \Gamma_{Tf}$ under the identification $T(M \times N) \simeq TM \times TN$.
We denote by
\begin{equation}\label{eq:lambdaf1}
\Lagr{f} \coloneqq T^*_{\Gamma_f}(M \times N)
\end{equation}
the conormal bundle of the graph of $f$.
The fiberwise transpose of $f'$ is denoted by $f_d$.
The restrictions of the projections $p_{T^*M} \colon T^*(M \times N) \to T^*M$ and $p_{T^*N} \colon T^*(M \times N) \to T^*N$ to $\Lagr{f}$ will be denoted by $p_M$ and $p_N$ respectively.
We write $p_M^a \coloneqq a_M \comp p_M$.
The map $\pi_M \times \id_{T^*N} \colon T^*(M \times N) \simeq T^*M \times T^*N \to M \times T^*N$ induces an isomorphism $\Lagr{f} \isoto M \times_N T^*N$.
We have the following commutative diagram.
\begin{equation}\label{eq:lambdaf2}
\xymatrix{
& \Lagr{f} \ar[ld]_{p_M^a} \ar[rd]^{p_N}\\
T^*M \ar[dr]_{\pi_M} & M \times_N T^*N \ar[u]^\wr \ar[d]^{p_M} \ar[l]_-{f_d} \ar[r]^-{f_\pi} & T^*N \ar[d]^{\pi_N}\\
& M \ar[r]^f & N
}
\end{equation}

For $A \subseteq T^*M$ and $B \subseteq T^*N$, one has
\begin{equation}
f_\pi \opb{f_d}(A) = A \acomp \Lagr{f}
\qquad\text{and}\qquad
f_d \opb{f_\pi}(B) = \Lagr{f} \acomp B.
\end{equation}

Note that $f$ is a \define{submersion} if and only if
\begin{equation}\label{eq:subm}
\Lagr{f} \cap (0^*_M \times T^*N) \subseteq 0^*_{M \times N}.
\end{equation}
More generally, if $B \subseteq T^*N$ is a closed cone, one says that $f$ is \define{non-characteristic for $B$} if
\begin{equation}\label{eq:non-char}
\Lagr{f} \cap (0^*_M \times B)\subseteq 0^*_{M \times N}.
\end{equation}

\subsection{Sheaves}

We recall in this and the next two subsections a few basic results on sheaves, kernels, and their microlocal theory, and we refer to~\cite{KS} for a complete treatment.

The \define{support} of a presheaf is the complement of the union of the open subsets sent by this presheaf to zero; in particular, it is a closed subset.
Let $X$ be a topological space.
We denote by $\cor_X$ the \define{constant sheaf on $X$ associated with $\cor$}.
If $f \colon X \to Y$ is a continuous map, then $\opb{f}(\cor_Y) = \cor_X$.

In the rest of this subsection, $A, A_1, A_2$ (resp.\ $B$) denote locally closed subsets of the topological space $X$ (resp.\ $Y$), and $i$ will denote the inclusion of that subset.
We define the \define{constant sheaf on $A$ associated with $\cor$ extended by zero to~$X$} by $\cor_{X, A} \coloneqq \eim{i} \opb{i} \cor_X$.
We will also denote it by $\cor_A$ if there is no risk of confusion.
It is characterized by $\cor_{X, A}|_A = \cor_X|_A$ and $\cor_{X, A}|_{X \setminus A} = 0$.
If $A_1, A_2 \subseteq X$, then
\begin{equation}\label{eq:prodExtended}
\cor_{X, A_1} \otimes \cor_{X, A_2} = \cor_{X, A_1 \cap A_2}.
\end{equation}

Let $\Derb(\cor_X)$ denote the \define{bounded derived category} of the category of sheaves of $\cor$-modules on~$X$.
Its objects will still be called sheaves.
For $F \in \Derb(\cor_X)$ and $A \subseteq X$, one sets
\begin{align}
F_A &\coloneqq F \Ltens \cor_{X, A}
\qquad\qquad\text{and}\\
\rsect_A (F) &\coloneqq \rhom(\cor_{X, A}, F).
\end{align}
The functor $(-)_A$ is exact.

By the Grothendieck spectral sequence, the derived functors of $\Gamma(X; -) \fcomp (-)_A$ and of $\Gamma_A(X; -) \coloneqq \Gamma(X; -) \fcomp \Gamma_A(-)$ are respectively
\begin{align}
\label{eq:sectFZ}
\rsect(X; (-)_A) &\simeq \rsect(X; -) \fcomp (-)_A
\qquad\text{and}\\
\label{eq:sectGammaZ}
\rsect_A(X; -) &\simeq \rsect(X; -) \fcomp \rsect_A(-).
\end{align}

If $Z$ is closed and $U$ is open in $X$ and $i$ denotes either inclusion, then, for a genuine sheaf $F$, one has $F_Z = \oim{i} \opb{i} F$ and $\Gamma_U F = \oim{i} \opb{i} F$.
Therefore, $\Gamma(X; F_Z) = \Gamma(Z; F)$ and $\Gamma(X; \Gamma_U(F)) = \Gamma(U; F)$.
Therefore, in the derived category,
\begin{align}
\label{eq:rsectClosed}
\rsect(X; F_Z) &\simeq \rsect(Z; F) \qquad\text{and}\\
\label{eq:rsectOpen}
\rsect(X; \rsect_U(F)) &\simeq \rsect(U; F).
\end{align}

If $f \colon X \to Y$ is continuous and $G \in \Derb(\cor_Y)$ and $B \subseteq Y$, then $\opb{f}(G_B) = (\opb{f}G)_{\opb{f}(B)}$ (\cite{KS}*{(2.3.19)}), and in particular,
\begin{equation}\label{eq:opbSheaf}
\opb{f}(\cor_{Y, B}) = \cor_{X, \opb{f}(B)}.
\end{equation}

In the rest of this subsection, assumptions are made on the topological spaces involved (hausdorffness, local compactness, finite c-soft dimension) and their morphisms (finite cohomological dimension).
All of these properties are satisfied by topological manifolds and their morphisms.

We recall the following fundamental result without proof.

\begin{proposition}[proper base change \cite{KS}*{Prop.~2.6.7,~3.1.9}]\label{prop:baseChange}
Let $f \colon X \to Z$ and $g \colon Y \to Z$ be continuous maps between Hausdorff locally compact spaces.
Denote their pullback as follows:
\begin{equation}
\begin{gathered}
\xymatrix{
T \ar @{}[dr]|<{\big\lrcorner} \ar[r]^q \ar[d]_p & Y \ar[d]^g\\
X \ar[r]^f & Z
}
\end{gathered}.
\end{equation}
Then, one has a natural isomorphism of functors
\begin{equation}
\reim{q} \fcomp \opb{p} \cong \opb{g} \fcomp \reim{f}.
\end{equation}
If $\eim{g}$ has finite cohomological dimension, then so does $\eim{p}$, and one has a natural isomorphism of functors
\begin{equation}
\roim{q} \fcomp \epb{p} \cong \epb{g} \fcomp \roim{f}.
\end{equation}
\end{proposition}

If $f \colon X \to Y$ is a continuous map of finite cohomological dimension between topological spaces, then the \define{relative dualizing complex} (\cite{KS}*{Def.~3.1.16(i)}) of $f$ is $\omega_{X/Y} \coloneqq \epb{f} \cor_Y$.
There is a natural transformation $\opb{f}(-) \Ltens \omega_{Y/X} \Rightarrow \epb{f}(-)$.
It is an isomorphism if $f$ is a topological submersion between Hausdorff locally compact spaces (\cite{KS}*{Prop.~3.3.2(ii)}) and under microlocal conditions that we give in Propositions~\ref{prop:funcSS}(2.b) and~\ref{prop:funcSSCont}(2).
We write $\omega_X \coloneqq \omega_{X/ \{ \ast \}}$.
If $X$ is a $C^0$-manifold, then $\omega_X$ is isomorphic to the orientation sheaf shifted by the dimension of~$X$.

Now, assume that $X$ is Hausdorff locally compact and has finite c-soft dimension.
The \define{duality functors} on $X$ (see~\cite{KS}*{Def.~3.1.16(ii)}) are defined by
\begin{align}
\RD_X &\coloneqq \rhom(-, \omega_X)
\qquad\text{and}\\
\RD'_X &\coloneqq \rhom(-, \cor_X).
\end{align}
Recall that if $F \in \Derb(\cor_X)$ is cohomologically constructible (see~\cite{KS}*{Def.~3.4.1}), then so are $\RD_X(F)$ and $\RD'_X(F)$, and $F \simeq \RD_X(\RD_X(F)) \simeq \RD'_X(\RD'_X(F))$ (this is~\cite{KS}*{Prop.~3.4.3}), that is, cohomologically constructible sheaves are reflexive.

\subsection{Kernels}

Let $X_i$ be Hausdorff locally compact spaces for $i \in \{ 1, 2, 3 \}$.
A \define{kernel} from $X_1$ to $X_2$ is an object of $\Derb(\cor_{X_{12}})$.
We consider the bifunctor of \define{convolution of kernels} (\cite{KS}*{Prop.~3.6.4}) defined on objects by
\begin{align}
\conv \colon \Derb(\cor_{M_{12}}) \times \Derb(\cor_{M_{23}}) \longrightarrow & \Derb(\cor_{M_{13}})\notag\\
(K_1,K_2) \longmapsto & K_1\conv K_2 \coloneqq \reim{p_{13}} \left( \opb{p_{12}}(K_1) \Ltens \opb{p_{23}}(K_2) \right)
\end{align}
and similarly on morphisms.
If there is a risk of confusion, we write this convolution as $\conv[2]$.
Recall that with our notation, $p_i$, for instance, can stand for $p_{M_i}$ or $p_{TM_i}$ or $p_{T^*M_i}$, and the signification is clear from the context, for instance in the above formula.

The convolution of kernels is associative.
We spell out the special cases when either $X_1$ or $X_3$ is a point.
Adapting the notation, for $K \in \Derb(\cor_{M \times N})$ and $F \in \Derb(\cor_M)$ and $G \in \Derb(\cor_N)$, one has
\begin{equation}
F \conv K \simeq \reim{p_N}(\opb{p_M}F \Ltens K)
\qquad\text{and}\qquad
K \conv G \simeq \reim{p_M}(K \Ltens \opb{p_N}G).
\end{equation}

The following two standard results show that the convolution of kernels is a generalization of the composition of functions.

If $f \colon X \to Y$ is a function between topological spaces, we set
\begin{equation}\label{eq:sheafFunc}
K_f \coloneqq \cor_{X \times Y, \Gamma_f}
\end{equation}
for the constant sheaf on the graph of $f$ associated with $\cor$ extended by zero to $X \times Y$.

\begin{proposition}\label{prop:kerFunc}
Let $f \colon X \to Y$ be a continuous map between Hausdorff locally compact spaces.
\begin{enumerate}
\item
If $F \in \Derb(\cor_X)$, then $F \conv K_f \simeq \reim{f}F$.
\item
If $G \in \Derb(\cor_Y)$, then $K_f \conv G \simeq \opb{f}G$.
\end{enumerate}
\end{proposition}

\begin{proof}
This follows easily from the proper base change theorem applied to the pullback of~$f$ and $\id_Y$.
\end{proof}

\begin{proposition}\label{prop:compFunc}
If $f_i \colon X_i \to X_{i+1}$ for $i \in \{1, 2\}$ are continuous maps between Hausdorff locally compact spaces, then
\begin{equation}
K_{f_2 \fcomp f_1} = K_{f_1} \conv K_{f_2}.
\end{equation}
\end{proposition}

\begin{proof}
This follows easily from Equations~\eqref{eq:opbSheaf} and~\eqref{eq:prodExtended} and the fact that $p_{13}$ induces an isomorphism $\opb{p_{12}}(\Gamma_{f_1}) \cap \opb{p_{23}}(\Gamma_{f_2}) = \Gamma_{f_1} \times_{M_2} \Gamma_{f_2} \isoto[p_{13}] \Gamma_{f_1} \comp \Gamma_{f_2} = \Gamma_{f_2 \fcomp f_1}$.
\end{proof}

\subsection{Microsupport of sheaves}

Let $M$ be a manifold and let $F \in \Derb(\cor_M)$.
We define the \define{propagation set} of $F$ by
\begin{multline}\label{eq:propagate}
\Ppg(F) \coloneqq \big\{ (x, \xi) \in T^*M \mid
\text{for all $\phi \in C^\infty(U)$, where $U$ is an open neighborhood of $x$,}\\
\text{with $\phi(x) = 0$ and $d\phi(x) = \xi$, one has $(\rsect_{\{ \phi \geq 0 \}}(F))_x = 0 $} \big\}.
\end{multline}
We set $\Ppg_x(F) \coloneqq \Ppg(F)_x$.
The \define{microsupport} of $F$ (see~\cite{KS}*{Def.~5.1.2(i)}) is defined by
\begin{equation}\label{eq:microsupp}
\muSupp(F) \coloneqq \closure{T^*M \setminus \Ppg(F)}.
\end{equation}
Therefore, the microsupport of a sheaf is the closure of the set of codirections of nonpropagation.
It is a coisotropic closed conic subset of $T^*M$ such that $\pi_M(\muSupp(F)) = \supp(F)$ and $\muSupp(F[i]) = \muSupp(F)$ for $i \in \Z$, and satisfies the following triangular inequality: if $F_1 \to F_2 \to F_3 \to[+1]\;$ is a distinguished triangle in $\Derb(\cor_M)$ and $i, j, k \in \{1, 2, 3\}$ with $j \neq k$, then $\muSupp(F_i) \subseteq \muSupp(F_j) \cup \muSupp(F_k)$ (see~\cite{KS}*{Prop.~5.1.3 and Thm.~6.5.4}).
We set $\muSupp_x(F) \coloneqq \muSupp(F)_x$.

The following lemma gives a useful criterion for belonging to the propagation set of a sheaf.

\begin{lemma}\label{lem:criterSheaf}
Let $M$ be a manifold, let $F \in \Derb(\cor_M)$ and let $x \in M$.
Then, $\xi \in \Ppg_x(F)$ if and only if for all $\phi \in C^\infty(U)$, with $U$ an open neighborhood of~$x$, such that $\phi(x) = 0$ and $d\phi(x) = \xi$, the morphism
\begin{equation}
F_x \longrightarrow (\rsect_{\{ \phi < 0 \}}(F))_x
\end{equation}
induced by the inclusions $(i_V \colon V \cap \{ \phi < 0 \} \hookrightarrow V)_{V \ni x}$ is an isomorphism.
\end{lemma}

\begin{proof}
The result follows by applying the stalk functor to the distinguished triangle $\rsect_{\{ \phi \geq 0 \}}(F) \to \rsect_U(F) \to \rsect_{\{ \phi < 0 \}}(F) \to[+1]\;$.
\end{proof}

The following lemma will simplify some proofs below.
The definition of strict differentiability is recalled in Appendix~\ref{app:strictDiff}.

\begin{lemma}\label{lem:SSstrictDiff}
Let $M$ be a manifold and let $F \in \Derb(\cor_M)$.
If $(x, \xi) \notin \muSupp(F)$, then for all functions $\phi \colon U \to \R$, with $U$ an open neighborhood of~$x$, which are strictly differentiable at~$x$ with $d\phi(x) = \xi$, one has $(\rsect_{\{ \phi \geq \phi(x) \}}(F))_x = 0$.
\end{lemma}

\begin{proof}
In the proof of~\cite{KS}*{Prop.~5.1.1}, the part (2)$\Rightarrow$(1)$_1$ proves the lemma without any change.
Indeed, with the notation there, if $\phi \colon U \to \R$ is strictly differentiable at~$x$ with $d\phi(x) \in \interior{\gamma^{\circ a}}$, then $\{ \phi < 0 \}$ coincides with a $\gamma$-open set in a neighborhood of~$x$ by Lemma~\ref{lem:strict-diff}.
\end{proof}

\begin{remark}
The analogous statement assuming only differentiability of $\phi$ is false.
Indeed, let $Z = \{ (x, y) \in \R^2 \mid x + y = 0 \text{ or } \exists n \in \N_{> 0} \; x + y = 1/n \}$.
Then, $\muSupp(\cor_Z) = Z \times \R(1, 1)$.
Let $\phi \colon \R^2 \to \R, (x, y) \mapsto 2 x^2 \sin(\pi/x) + y$.
Then, $\phi(0, 0) = 0$ and $d\phi(0, 0) = (0, 1) \notin \muSupp_{(0, 0)}(\cor_Z)$.
For $n \in \N_{> 0}$, set $U_n = \intervO{-1/n, 1/n}^2$.
Then, the inclusions $i_n \colon U_n \cap Z \cap \{ \phi < 0 \} \hookrightarrow U_n \cap Z$ do not induce an isomorphism in cohomology of the inductive limit.
Indeed, setting $Z_m \coloneqq \{ x + y = 1/m \}$, the subsets $Z_m \cap U_n$ are connected, but for any $n, m \in \N$ with $2m > n$, the subset $Z_{2m} \cap U_n \cap \{ \phi < 0 \}$ is not.
\end{remark}

\begin{proposition}[\cite{KS}*{Exe.~V.13}]\label{prop:dual}
Let $M$ be a manifold and let $F \in \Derb(\cor_M)$ be cohomologically constructible.
Then,
\begin{equation}
\muSupp(\RD'_M(F)) = \muSupp(F)^a.
\end{equation}
\end{proposition}

\begin{proof}
Since constructible sheaves are reflexive, it suffices to prove $\muSupp(\RD'_M(F)) \subseteq \muSupp(F)^a$.
This is a special case of~\cite{KS}*{Prop.~5.4.2} where one factor is reduced to a point.
\end{proof}

We recall the following functoriality properties of the microsupport.
A morphism is said to be non-characteristic for a sheaf if it is non-characteristic for its microsupport.

\begin{proposition}[\cite{KS}*{Prop.~5.4.4-5,~5.4.13-14}]\label{prop:funcSS}
Let $M$ be a manifold.
\begin{enumerate}
\item
Let $F_1, F_2 \in \Derb(\cor_M)$.
\begin{enumerate}
\item
Assume that $\muSupp(F_1) \cap \muSupp(F_2)^a \subseteq 0^*_M$.
Then, $\muSupp(F_1 \Ltens F_2) \subseteq \muSupp(F_1) + \muSupp(F_2)$.
\item
Assume that $\muSupp(F_1) \cap \muSupp(F_2) \subseteq 0^*_M$.
Then, $\muSupp(\rhom(F_1, F_2)) \subseteq \muSupp(F_1)^a + \muSupp(F_2)$.
\end{enumerate}
\item
Let $f \colon M \to N$ be a morphism of manifolds.
\begin{enumerate}
\item
Let $F \in \Derb(\cor_M)$ and assume that $f$ is proper on $\supp(F)$.
Then, $\muSupp(\roim{f} F) \subseteq \muSupp(F) \acomp \Lagr{f}$ with equality if $f$ is a closed embedding.
\item
Let $G \in \Derb(\cor_{N})$ and assume that $f$ is non-characteristic for $G$.
Then, the morphism $\opb{f} G \Ltens \omega_{M/N} \to \epb{f} G$ is an isomorphism and $\muSupp(\opb{f} G) \subseteq \Lagr{f} \acomp \muSupp(G)$ with equality if $f$ is a submersion.
\end{enumerate}
\end{enumerate}
\end{proposition}

\begin{remark}\label{rmk:funcPpg}
The proof of the inclusion in Item~(2.a) given in~\cite{KS}*{Prop.~5.4.4} actually proves that if $f \colon M \to N$ is a morphism of manifolds and $F \in \Derb(\cor_M)$, and if $f$ is proper on $\supp(F)$, then $TN \setminus \Ppg(\roim{f} F) \subseteq (TM \setminus \Ppg(F)) \acomp \Lagr{f}$ with equality if $f$ is an isomorphism.
\end{remark}

\begin{remark}\label{rem:funcSSstrictDiff}
By Lemma~\ref{lem:SSstrictDiff}, the inclusion of Item~(2.a) and the result of the previous remark still hold at a point if the map $f$ is only required to be strictly differentiable at that point.
\end{remark}

Finally, we give a standard upper bound on the microsupport of the convolute of two kernels.

\begin{proposition}[\cite{GS}*{\S~Kernels}]\label{prop:sskernels}
Let $M_i$ be manifolds for $i \in \{ 1, 2, 3 \}$.
Let $K_1 \in \Derb(\cor_{M_{12}})$ and $K_2 \in \Derb(\cor_{M_{23}})$.
Assume that
\begin{enumerate}
\item
$p_{13}$ is proper on $\opb{p_{12}}(\supp(K_1)) \cap \opb{p_{23}}(\supp(K_2))$,
\item
$(\muSupp(K_1)^a \times 0^*_3) \cap (0^*_1 \times \muSupp(K_2)) \subseteq 0^*_{123}$.
\end{enumerate}
Then,
\begin{equation}
\muSupp(K_1\conv K_2) \subseteq \muSupp(K_1) \acomp \muSupp(K_2).
\end{equation}
\end{proposition}

\begin{proof}
The two assumptions of the proposition allow to apply Proposition~\ref{prop:funcSS}(1.a), (2.a) and (2.b) to conclude.
\end{proof}

\begin{remark}
Note that the two assumptions of Proposition~\ref{prop:sskernels} are equivalent to the following condition: $p_{13}$ is proper on $\opb{p_{12^a}}(\muSupp(K_1)) \cap \opb{p_{23}}(\muSupp(K_2))$.
\end{remark}

\section{Microsupports associated with subsets}
\label{sec:subsets}

In this section, we study the microsupports of (constant sheaves on) subsets.
We first give a criterion for belonging to such microsupports and give results on the microsupports of closed submanifolds and closed $C^0$-submanifolds.
We then offer direct proofs of lower and upper bounds on the microsupport of a closed set.

\subsection{General properties}

Let $M$ be a manifold.
If $A \subseteq M$ is locally closed, we set for short $\Ppg(A) \coloneqq \Ppg(\cor_{M, A})$ and
\begin{equation}
\muSupp(A) \coloneqq \muSupp(\cor_{M, A}),
\end{equation}
and also $\Ppg_x(A) \coloneqq \Ppg(A)_x$ and $\muSupp_x(A) \coloneqq \muSupp(A)_x$ for $x \in M$.
Note that
\begin{equation}\label{eq:SSproj}
\pi_M (\muSupp(A)) = \closure{A}.
\end{equation}

The following lemma gives a useful criterion for belonging to the propagation set of a sheaf associated with a subset.

\begin{lemma}\label{lem:criterSet}
Let $M$ be a manifold, let $Z \subseteq M$ be a closed subset and let $x \in M$.
Then, $\xi \in \Ppg_x(Z)$ if and only if for all $\phi \in C^\infty(U)$, with $U$ an open neighborhood of~$x$, such that $\phi(x) = 0$ and $d\phi(x) = \xi$, the morphism
\begin{equation}
(\cor_Z)_x \longrightarrow (\rsect_{Z \cap \{ \phi < 0 \}} \cor_Z)_x
\end{equation}
induced by the inclusions $(i_V \colon Z \cap \{ \phi < 0 \} \cap V \hookrightarrow Z \cap V)_{V \ni x}$ is an isomorphism.
\end{lemma}

\begin{proof}
As for Lemma~\ref{lem:criterSheaf}, the result follows by applying the stalk functor to the distinguished triangle $\rsect_{Z \cap \{ \phi \geq 0 \}}(\cor_Z) \to \rsect_{Z \cap U}(\cor_Z) \to \rsect_{Z \cap \{ \phi < 0 \}}(\cor_Z) \to[+1]\;$.
\end{proof}

As for inverse images, Proposition~\ref{prop:funcSS}(2.b) and Equation~\eqref{eq:opbSheaf} show that if $f \colon M \to N$ is a morphism of manifolds and $B \subseteq N$ is locally closed, and if $f$ is non-characteristic for $\muSupp(B)$, then
\begin{equation}\label{eq:opbSet}
\muSupp(\opb{f}(B)) \subseteq \Lagr{f} \acomp \muSupp(B)
\end{equation}
with equality if $f$ is a submersion.
If $f$ is an isomorphism and $A \subseteq M$ is locally closed, then Remark~\ref{rmk:funcPpg} gives
\begin{equation}\label{eq:opb-Ppg}
\Ppg(f(A)) = \Ppg(A) \acomp \Lagr{f}.
\end{equation}
In particular, if $\Phi \colon \V \to \V'$ is a linear isomorphism, $A \subseteq \V$ is locally closed and $x \in \V$, then
\begin{equation}\label{eq:opb-Ppg-linear}
\Phi^\intercal(\Ppg_{\Phi(x)}(\Phi(A))) = \Ppg_x(A).
\end{equation}

We will also need the following result.

\begin{proposition}[\cite{KS}*{Prop.~5.3.2}]\label{prop:subman}
Let $N$ be a closed submanifold of a manifold $M$.
Then
\begin{equation}
\muSupp(N) = T^*_NM.
\end{equation}
\end{proposition}

For closed $C^0$-submanifolds, we have the following result.

\begin{proposition}\label{prop:symmetric}
Let $N$ be a closed $C^0$-submanifold of a manifold $M$.
Then, $\cor_N$ is a cohomologically constructible sheaf, self-dual up to a locally constant sheaf of rank 1, and its microsupport is symmetric, that is,
\begin{equation}
\muSupp(N) = \muSupp(N)^a.
\end{equation}
\end{proposition}

\begin{proof}
The cohomological constructibility and self-duality results hold if $N$ is a vector subspace of a vector space $M$ (in that case, $\RD_M(\cor_{M, N}) = \rhom(\cor_{M, N}, \omega_M) = \omega_N$ by \cite{KS}*{Exa.~3.4.5(i)}), and they are of a local and topological nature.
Therefore, they hold for closed $C^0$-submanifolds, and we can apply Proposition~\ref{prop:dual}.
\end{proof}

\subsection{Bounds on the microsupports associated with closed subsets}

We have the following bounds on the microsupport associated with a closed subset.
For the tangent and strict tangent cones appearing in this proposition, we refer to Appendix~\ref{app:cones}.

\begin{proposition}\label{prop:SSbounds}
Let $M$ be a manifold.
If $Z \subseteq M$ is closed, then
\begin{equation}\label{eq:SSbounds}
\closure{\opb{\pi_M}(Z) \cap C(Z)^\circ} \subseteq \muSupp(Z) \subseteq N(Z)^\circ.
\end{equation}
\end{proposition}

The upper bound is~\cite{KS}*{Prop.~5.3.8}.
The lower bound was proved in~\cite{KMS}*{Prop.~3.1}, where it was proved that $\closure{\opb{\pi_M}(Z) \cap C(Z)^\circ}$ is equal to the 0-truncated microsupport of $\cor_Z$ (see definition there).
We give direct proofs of both bounds.

The lower bound is not widely applicable, since $C_x(Z)^\circ$ is nonzero only if $Z$ is ``at first order'' contained in a half-space of $T_x M$.
On the other hand, the upper bound is trivial if $Z$ is a $C^0$-submanifold.
We will give another upper bound in that case (Proposition~\ref{prop:upperBoundSubman}).

For the proof, we will need the following two lemmas.

\begin{lemma}\label{lem:homotopy}
Let $M$ be a manifold, $A \subseteq M$ and $x \in M$.
Let $U$ be an open neighborhood of~$x$ and $\phi \in C^\infty(U)$ with $\phi(x) = 0$ and $d\phi(x) \notin N_x(A)^\circ$.
Then, there exist an open neighborhood $V \subseteq U$ of~$x$ and $\alpha > 0$ and for all $\beta \in \intervO{-\alpha, \alpha}$, a strong deformation retraction $h_\beta \colon V \times [0, 1] \to V$ of $V$ onto $V \cap \{ \phi \leq \beta \}$ such that $h_\beta( (V \cap A) \times [0, 1]) \subseteq V \cap A$.
\end{lemma}

\begin{proof}
Let $u \in N_x(A)$ be such that $\innerp{d\phi(x)}{u} < 0$.
We can fix a chart at~$x$ with domain $V \subseteq U$ such that in that chart, $x = 0 \in \R^m$ and $\phi = d\phi(0)$ (since $d\phi(0) \neq 0$).
Since $N(A)$ is open, we can suppose, reducing $V$ if necessary, that for all $y \in V$, one has $u \in N_y(A)$.
We can also suppose that the intersection of $V$ with any line parallel to $u$ is connected, for instance by assuming that $V$ is of the form $\intervO{-\alpha, \alpha} \times V'$ and $u \in \R \times \{0\}$.

If $\beta \in \intervO{-\alpha, \alpha}$ and $y \in V$ and $t \in [0, 1]$, we set $h_\beta(y, t) \coloneqq y - t \frac{(\phi(y) - \beta)^+}{\phi(u)} u$ where $x^+ \coloneqq \max(x, 0)$.
This defines $h_\beta \colon V \times [0, 1] \to V$.
If $\phi(y) \leq \beta$ and $t \in [0, 1]$, then $h_\beta(y, t) = y$.
If $y \in V$, then $h_\beta(y, 0) = y$ and $\phi(h_\beta(y, 1)) = \phi(y) - (\phi(y) - \beta)^+ = \min( \phi(y), \beta) \leq \beta$.
Finally, if $y \in V \cap A$ and $t \in [0, 1]$, then the assumptions $-t \frac{(\phi(y) - \beta)^+}{\phi(u)} \geq 0$ and $u \in N_{h_\beta(y,s)}(A)$ for any $s \in [0, t]$ imply $h_\beta(y, t) \in A$.
\end{proof}

\begin{lemma}\label{lem:cone}
Let $A$ be a subset of a manifold $M$.
Then,
\begin{multline}
C(A)^\circ = \{ (x, \xi) \in \closure{A} \times_M T^*M \mid \text{there exist an open neighborhood $U$ of $x$ and}\\
\phi \in C^1(U) \text{ such that } \phi(x) = 0 \text{ and } d\phi(x) = \xi \text{ and } \phi(A \cap U) \subseteq \R_{\geq 0} \}
\end{multline}
and this set is equal to the set defined similarly with the function~$\phi$ only required to be continuous on~$U$ and differentiable at~$x$.
\end{lemma}

\begin{proof}
(i)
Let $(x, \xi) \in \closure{A} \times_M T^*M$.
Let $U$ be an open neighborhood of $x$ and $\phi \in C^0(U)$ be such that $\phi(x) = 0$ and $d\phi(x) = \xi$ and $\phi(A \cap U) \subseteq \R_{\geq 0}$.
Let $u \in C_x(A)$.
There are sequences $(x_n) \in A^\N$ and $(c_n) \in (\R_{>0})^\N$ such that $x_n \to[n]x$ and $c_n (x_n - x) \to[n] u$.
One has $0 \leq c_n (\phi(x_n) - \phi(x)) = \innerp{d\phi(x)}{c_n (x_n - x)} + o(c_n (x_n - x))$.
Therefore, $\innerp{\xi}{u} \geq 0$.
Therefore, $(x, \xi) \in C(A)^\circ$.

\spa
(ii)
Conversely, let $(x, \xi) \in C(A)^\circ$.
Fix a chart centered at~$x$, with image in $\R^m$.
If $m = 0$, the result is trivial, so we suppose $m > 0$.
Denote by $\norm{-}$ the Euclidean norm on $\R^m$.
If $\xi = 0$, then set $\phi \coloneqq \norm{-}^2 \colon \R^m \to \R$.
Now, suppose $\xi \neq 0$.
We can suppose that $\xi = (-1, 0, \dots, 0)$.

Let $n \in \N_{> 0}$.
We will prove that there exists $\alpha_n > 0$ such that, setting $E_n \coloneqq \{ (u_1, u') \in \R \times \R^{m-1} \mid \frac1n \norm{u'} \leq u_1 \leq \alpha_n \} \setminus \{0\}$, one has $E_n \cap A = \varnothing$.
Indeed, suppose that there is a sequence of points $(x_m) = \big( (x_{m, 1}, x'_m) \big) \in A^\N$ such that $x_{m, 1} \geq \frac1n \norm{x'_m} > 0$ for all $m \in \N$ and $x_{m, 1} \to[m] 0$.
Then, $x_m \to[m] 0$, so up to extracting a subsequence, we can suppose that $\frac{x_m}{\norm{x_m}} \to[m] u \in \unp{C}_0(A)$.
Since $\xi \in C_x(A)^\circ$, one has $\innerp{\xi}{u} \geq 0$.
But, writing $u = (u_1, u')$, one has $\innerp{p}{u} = - u_1 \leq -\frac1n \norm{u'} \leq 0$, and actually, $\innerp{\xi}{u} < 0$, since otherwise, one would have $u_1 = u' = 0$, so $u = 0$.
This is a contradiction.
This proves the existence of the desired $\alpha_n > 0$.
We can suppose that for $n \in \N_{> 0}$, one has $(n+1)\alpha_{n+2} < n \alpha_{n+1}$, and in particular the sequence $(\alpha_n)$ decreases to 0.

We can assume that the broken line connecting the points $A_n = (n \alpha_{n+1}, \alpha_{n+1}) \in \R^2$ defines a convex function $f \colon \intervCO{0, \alpha_2} \to \R$.
Indeed, each line $(A_n A_{n+1})$ crosses the $x$-axis at some $\lambda_n > 0$, and one can ensure recursively that $(n+2) \alpha_{n+2} \leq \lambda_n$.
It is elementary to construct a function $\psi \in C^1\big( \intervO{-\alpha_2, \alpha_2} \big)$ with $\psi(t) \geq f(\abs{t})$ for $t \in \intervO{-\alpha_2, \alpha_2}$ and $\psi(0) = \psi'(0) = 0$.
For instance, construct smooth functions with graphs in the triangles formed by $A_{n+1}$ and the midpoints of $[A_nA_{n+1}]$ and $[A_{n+1}A_{n+2}]$, which connect in a $C^1$ fashion.
Finally, setting $\phi \colon B(0, \alpha_2) \to \R, (x_1, x') \mapsto \psi(\norm{x'}) - x_1$, one has $\phi(0) = 0$ and $d\phi(0) = \xi$ and $\phi(A \cap B(0, \alpha_2)) \subseteq \R_{\geq 0}$.
\end{proof}

\begin{remark}
In view of Lemma~\ref{lem:SSstrictDiff}, we will only need in the applications of the lemma that the function $\phi$ constructed in the proof is strictly differentiable at~$x$, so we could have simply defined $\psi(t) \coloneqq f(\abs{t})$.
\end{remark}

\begin{remark}\label{rmk:cone}
In the right-hand side in the lemma, we cannot require that $\phi \in C^2(U)$.
Consider for example the graph $\Gamma_f$ of the function $f \colon \R \to \R, x \mapsto \abs{x}^{3/2}$.
Then, $(0, 1) \in C_{(0, 0)}(\Gamma_f)^\circ$, but the Taylor--Lagrange formula shows that for any neighborhood $U \subseteq \R$ of 0, there cannot be a function $\phi \in C^2(U)$ with $\phi(0) = 0$ and $d\phi(0) = 0$ and $f|_U \leq \phi$.
As a consequence, we see that in the definition of the propagation set of a sheaf, the differentiability class of the test functions matters, as opposed to the situation for the microsupport.
Namely, in this example, $(0, 1) \in \Ppg_{(0, 0)}^2(\Gamma_f) \setminus \Ppg_{(0, 0)}^1(\Gamma_f)$, where the superscript denotes the differentiability class of the test functions.
\end{remark}

\begin{proof}[Proof of the proposition]
(i)
Upper bound.
Since $N(Z)$ is open, its polar cone is closed, so it is enough to prove that $T^*M \setminus \Ppg(Z) \subseteq N(Z)^\circ$.
If $x \notin Z$, then $\muSupp_x(Z) = \varnothing$, so the inclusion is true.
Let $x \in Z$ and $\xi \in T^*_xM \setminus N_x(Z)^\circ$.
Let $\phi \in C^\infty(U)$ with $U$ an open neighborhood of~$x$ be such that $\phi(x) = 0$ and $d\phi(x) = \xi$.
We will prove that the natural morphism $(\cor_Z)_x \longrightarrow (\rsect_{Z \cap \{ \phi < 0 \}} \cor_Z)_x$ is an isomorphism.
Then, by Lemma~\ref{lem:criterSet}, we will have $\xi \in \Ppg_x(Z)$.

By Lemma~\ref{lem:homotopy}, there exist an open neighborhood $V \subseteq U$ of~$x$ and $\alpha > 0$ and for all $\beta \in (-\alpha, 0)$, a deformation retraction $h_\beta \colon V \times [0, 1] \to V$ of $V$ onto $V \cap \{ \phi \leq \beta \}$ such that $h_\beta( (V \cap Z) \times [0, 1]) \subseteq V \cap Z$.

This proves that the inclusions $i_{\beta, V} \colon Z \cap V \cap \{ \phi \leq \beta \} \hookrightarrow Z \cap V$ induce isomorphisms ${i_{\beta, V}}^\sharp \colon \rsect(Z \cap V; \cor_M) \isoto \rsect(Z \cap V \cap \{ \phi \leq \beta \}; \cor_M)$.
One has $\rsect(Z \cap V \cap \{ \phi < 0 \}; \cor_M) = \lim_{\beta \to[<] 0} \rsect(Z \cap V \cap \{ \phi \leq \beta \}; \cor_M)$, and if $\beta < \beta'$, then $i_{\beta', V} \circ i_{\beta, V} = i_{\beta, V}$.
Therefore, the $i_{\beta, V}$'s induce an isomorphism which is ${i_V}^\sharp \colon \rsect(Z \cap V; \cor_M) \isoto \rsect(Z \cap V \cap \{ \phi < 0 \}; \cor_M)$.

Since in Lemma~\ref{lem:homotopy}, the homotopies $h_\beta$ can be restricted to arbitrarily small neighborhoods $V$ of~$x$, this proves that the natural morphism $(\cor_Z)_x \longrightarrow (\rsect_{Z \cap \{ \phi < 0 \}} \cor_Z)_x$ is an isomorphism.

\spa
(ii)
Lower bound.
Let $x \in Z$ and $\xi \in C_x(Z)^\circ$.
Then, by Lemma~\ref{lem:cone}, there exist an open neighborhood $U$ of~$x$ and a function $\phi \in C^1(U)$ such that $\phi(x) = 0$ and $d\phi(x) = \xi$ and $\phi(Z \cap U) \subseteq \R_{\geq 0}$.
Therefore, $Z \cap U \cap \{ \phi < 0 \} = \varnothing$, so $(\rsect_{Z \cap \{ \phi < 0 \}} \cor_Z)_x = 0 \not\simeq (\cor_Z)_x = \cor$, so by Lemma~\ref{lem:criterSet}, $\xi \notin \Ppg_x(Z)$.
\end{proof}

\begin{remark}
The proof actually shows that $\opb{\pi_M}(Z) \cap C(Z)^\circ \subseteq T^*M \setminus \Ppg^1(Z)$.
\end{remark}

\begin{remark}
Taking polars of the inclusions~\eqref{eq:SSbounds}, we see that the fiberwise closure of $N(Z)$ is contained in $\muSupp(Z)^\circ$, which is contained in $C(Z)$, but the inclusion $\closure{N(Z)} \subseteq \muSupp(Z)^\circ$ need not hold: if $x \in \partial(\interior{Z})$, then $\closure{N(Z)}_x = T_xM$ but $\muSupp_x(Z) \supsetneq \{0\}$.
\end{remark}

\section{Whitney cones of maps}
\label{sec:whitney}

In this section, we study the Whitney cones of continuous maps.
The results are elementary and their proofs do not require any sheaf theory.
After the definitions and general properties, we give characterizations of Lipschitz continuity and strict differentiability in terms of the Whitney cone and extend these characterizations to $C^0$-submanifolds.
Then, we prove a chain rule involving Whitney cones.
Finally, we introduce directional Dini derivatives and relate them to the Whitney cone, which will be used to prove an upper bound on the conormal of a continuous map in Section~\ref{sec:general}.

\subsection{Definitions and first properties}

The definitions of the strict tangent cone $N(A)$, the tangent cone $C(A)$, and the Whitney cone $C(A, B)$ of subsets $A, B$ of a manifold, and their main properties, are recalled in Appendix~\ref{app:cones}.
If $f \colon M \to N$ is a continuous map between manifolds, we define its \define{Whitney cone} as the Whitney cone of its graph, that is,
\begin{equation}
\Whit{f} \coloneqq C(\Gamma_f, \Gamma_f).
\end{equation}
This is a closed symmetric cone in $T(M \times N)$.

If furthermore $x \in M$, we set
\begin{equation}
\Whit[x]{f} \coloneqq C_{(x, f(x))} (\Gamma_f, \Gamma_f).
\end{equation}
This is a closed cone in $T_{(x, f(x))}(M \times N)$.
We also set $C_x(\Gamma_f) \coloneqq C_{(x, f(x))}(\Gamma_f)$ and similarly for $N_x(\Gamma^+_f)$.

If furthermore $u \in T_xM$, we define the following analogue of the directional derivatives:
\begin{equation}\label{eq:directionalWhitney}
\Whit[x][u]{f} \coloneqq \Whit[x]{f} \cap \big( \R_{\geq 0} u \times T_{f(x)}N \big).
\end{equation}

\subsection{Characterizations of Lipschitz continuity and strict differentiability}

We begin with two straightforward lemmas whose proofs are left to the reader.
For notions related to Lipschitz continuity and strict differentiability, we refer to Appendix~\ref{app:strictDiff}.
In particular, the notation $\Lip_x(f)$ is defined in Definition~\ref{def:lip}(3).

\begin{lemma}\label{lem:C-LipCone}
Let $f \colon \V \to \W$ be a map between normed vector spaces.
Let $x \in \V$.
Then,
\begin{equation}
\Lip_x(f) = \min \Big\{ C \in \overline{\R} \mathrel{\Big|} \Whit[x]{f} \subseteq \{ (v, w) \in \V \times \W \mid \norm{w} \leq C \norm{v} \} \Big\}.
\end{equation}
\end{lemma}

\begin{lemma}\label{lem:LipSubm}
If $f \colon M \to N$ is a continuous map between manifolds that is pointwise Lipschitz at $x \in M$, then $p_M(\Whit[x]{f}) = T_xM$.
\end{lemma}

The following proposition provides geometric characterizations of Lipschitz continuity and strict differentiability.

\begin{proposition}\label{prop:LipCone}
Let $f \colon M \to N$ be a continuous map between manifolds and let $x \in M$.
\begin{enumerate}
\item
The map $f$ is Lipschitz on a neighborhood of~$x$ if and only if $\Whit[x][0]{f} = \{0\}$.
\item
The map $f$ is strictly differentiable at~$x$ if and only if $\Whit[x][0]{f} = \{0\}$ and $\Whit[x]{f}$ is contained in a $(\dim_x M)$-dimensional vector subspace, and in that case, $\Whit[x]{f} = \Gamma_{T_x f}$.
\end{enumerate}
\end{proposition}

\begin{proof}
(i)
Fixing charts at~$x$ and $f(x)$, one has $\Whit[x][0]{f} = \{0\}$ if and only if there exists $C \in \R$ such that $\Whit[x]{f} \subseteq \{ (u, v) \in \R^{m+n} \mid \norm{v} \leq C \norm{u} \}$.
Therefore, the result follows from Lemma~\ref{lem:C-LipCone}.

\spa
(ii)
Necessity is obvious.
We prove sufficiency.
We deduce from $\Whit[x][0]{f} = \{0\}$ and~(i) that $f$ is Lipschitz at~$x$.
It follows from Lemma~\ref{lem:LipSubm} that $p_M(\Whit[x]{f}) = T_xM$.
Therefore, $\Whit[x]{f}$ is a $(\dim_x M)$-dimensional vector space and is the graph of the linear map $L \coloneqq p_N \circ \opb{(p_M|_{\Whit[x]{f}})} \colon T_xM \to T_{f(x)}N$.
From Lemma~\ref{lem:C-LipCone} applied to $f - L$, we see that $f$ is strictly differentiable at~$x$ with $T_xf = L$.
\end{proof}

These characterizations extend to $C^0$-submanifolds as follows.

\begin{proposition}\label{prop:submanWhit}
Let $M$ be a $C^0$-submanifold of a manifold $P$ and let $x \in M$.
\begin{enumerate}
\item
Let $F$ be a $(\codim_x M)$-dimensional vector subspace of $T_x P$.
The $C^0$-submanifold $M$ is locally at~$x$ the graph of a Lipschitz map with codomain tangent to $F$ if and only if $C_x(M, M) \cap F = \{0\}$.
\item
The $C^0$-submanifold $M$ is locally at~$x$ the graph of a map $f$ strictly differentiable at~$x$ if and only if $C_x(M, M)$ is contained in a $(\dim_x M)$-dimensional vector subspace of $T_x P$, and in that case, $C_x(M, M) = \Gamma_{T_xf}$.
\end{enumerate}
\end{proposition}

The condition that $M$ be locally at~$x$ a graph means that there exist an open neighborhood $U$ of~$x$ and a chart $\phi = (\phi_M, \phi_N) \colon U \isoto U_M \times U_N \subseteq \R^m \times \R^n$ and a map $f \colon U_M \to U_N$ such that $\phi(M \cap U) = \Gamma_f$.
The condition ``$f$ strictly differentiable at~$x$'' then means ``$f$ strictly differentiable at $\phi_M(x)$''.
One identifies $C_x(M, M)$ with its image in the chart $T_x \phi(C_x(M, M)) = C_{\phi(x)}(\phi(M \cap U), \phi(M \cap U))$.
That the ``codomain of $f$ is tangent to $F$'' means that $T_x\phi (F) = \{0\} \times \R^n$.

\begin{proof}
In both cases, necessity is straightforward.
As for sufficiency, fix a chart~$\phi$ of $P$ at~$x$.
We set $m \coloneqq \dim_x M$ and $n \coloneqq \codim_x M$.
For the first equivalence of the proposition, $F$ is given, and for the second equivalence, we let $F$ be a complement of a $(\dim_x M)$-dimensional vector subspace of $T_x P$ in which $C_x(M, M)$ is included.
We can suppose that the chart~$\phi$ is of the form $\phi = (\phi_M, \phi_N) \colon U \isoto U_M \times U_N \subseteq \R^m \times \R^n$ with $T_x\phi(F) = \{0\} \times \R^n$, hence $T_x \phi (C_x(M, M)) \cap (\{0\} \times \R^n) = \{0\}$.
Being a $C^0$-submanifold, $M$ is locally at~$x$ the image of a continuous injection $\tilde{f} \colon \R^m \to \R^{m+n}$ with $\tilde{f}(0) = x$ and $\tilde{f}(\R^m) = M \cap U'$ for $U' \subseteq P$ an open neighborhood of~$x$.
We can suppose that $U' = U$.

The relation $T_x\phi (C_x(M, M)) \cap (\{0\} \times \R^n) = \{0\}$ implies that $\phi_M|_{M \cap U} \colon M \cap U \to \R^m$ is injective in a neighborhood of~$x$ which we can suppose to be $\tilde{f}(\R^m) = M \cap U$.
Therefore, $\phi_M \circ \tilde{f} \colon \R^m \to \R^m$ is a continuous injection, so by the theorem of invariance of domain, it is a topological isomorphism onto its image, say $V$, which is open in $\R^m$.

Then, define $f \coloneqq (\phi_N \circ \tilde{f}) \circ \opb{(\phi_M \circ \tilde{f})} = \phi_N \circ \opb{(\phi_M|_{M \cap U})} \colon V \to \R^n$.
In particular, $\phi_N = f \circ \phi_M$ on $\tilde{f}(\R^m) = M \cap U$.
Therefore, $\phi(M \cap U) = \Gamma_f$.
Therefore, $C_x(M, M) = \Whit[\phi_M(x)]{f}$.
By Proposition~\ref{prop:LipCone}, the hypothesis of Item~1 (resp.\ Item~2) on $C_x(M, M)$ implies that $f$ is Lipschitz (resp.\ strictly differentiable at~0 with derivative 0).
\end{proof}

\begin{remark}
The analogous statements with the tangent cone $C_x(M)$ and pointwise Lipschitz continuity (resp.\ differentiability) are false, as $M = \Gamma_f$ where $f = \sqrt{\abs{-}} \colon \R \to \R$ shows: one has $C_0(\Gamma_f) = \{0\} \times \R_{\geq 0}$, which intersects $\R \times \{0\}$ trivially.
Another example is any wild enough curve contained in $\{ (x, y) \in \R^2 \mid \abs{y} \leq x^2 \}$.
Of course, one still has one implication.

Note that we did not assume that the embedding $\tilde{f}$ in the proof was locally flat.
The local flatness of $M$ at~$x$ is a consequence of the hypothesis on $C_x(M, M)$.

If $\codim_x M = 1$, then the condition of Item~1 is equivalent to $C_x(M, M) \neq T_x P$.
\end{remark}

\begin{corollary}
Let $f \colon M \to N$ be a continuous map between manifolds and let $x \in M$.
If $\Whit[x]{f}$ is contained in a $(\dim_x M)$-dimensional vector subspace, then it is equal to it.
\end{corollary}

In particular, if $\dim_x M > 0$, then $\Whit[x]{f}$ cannot be contained in a $(\dim_x M - 1)$-dimensional vector subspace.
Contrast this to the case of $C_x(\Gamma_f)$: if $f = \sqrt{\norm{-}} \colon \R^m \to \R$, then $C_0(\Gamma_f) = \{0\} \times \R_{\geq 0}$ is contained in a 1-dimensional vector subspace.

We end this subsection with a result proving the lower hemicontinuity of $\Whit[x]{f}$ considered as a multivalued function from $T_xM$ to $T_{f(x)}N$.

\begin{proposition}\label{prop:hemicont}
Let $f \colon M \to N$ be a Lipschitz map between manifolds and let $x \in M$.
For any $u \in T_xM$ and any open set $V \subseteq T_{f(x)} N$ such that $\Whit[x]{f} \cap (\{u\} \times V) \neq \varnothing$, there exists an open neighborhood $U \subseteq T_xM$ of $u$ such that for all $u' \in U$, one has $\Whit[x]{f} \cap (\{u'\} \times V) \neq \varnothing$.
\end{proposition}

\begin{proof}
Let $(u, v) \in \Whit[x]{f}$ and let $(u_n) \in (T_xM)^\N$ be a sequence converging to $u$.
We are going to construct a sequence $(v_n) \in (T_{f(x)} N)^\N$ converging to $v$ such that $(u_n, v_n) \in \Whit[x]{f}$, which proves the proposition.
Fix charts at $x$ and $f(x)$ in which $f$ is $C$-Lipschitz.
There exist sequences $x_n, y_n \to[n] x$ and $c_n > 0$ such that $c_n (y_n - x_n, f(y_n) - f(x_n)) \to[n] (u, v)$.
Set $y_{m, n} \coloneqq y_n + (u_m - u)/c_n$.
One has $c_n (y_{m, n} - x_n) \to[n] u_m$ for all $m \in \N$.

One has $\norm{c_n (f(y_{m, n}) - f(x_n)) - c_n (f(y_n) - f(x_n))} = \norm{c_n (f(y_{m, n}) - f(y_n))} \leq C \norm{u_m - u}$.
Therefore, $c_n (f(y_{m, n}) - f(x_n)) \to[m] c_n (f(y_n) - f(x_n))$ uniformly in $n$.
By Lipschitz continuity of $f$, for any $m \in \N$, the sequence $(c_n (f(y_{m, n}) - f(x_n)))_n$ has a converging subsequence, indexed by $\rho$, and we set $v_m \coloneqq \lim_n c_{\rho(n)} (f(y_{m, \rho(n)}) - f(x_{\rho(n)}))$.
Since the above convergence is uniform, one has $v_m \to[m] v$.
For any $m \in \N$, one has $c_{\rho(n)} (y_{m, \rho(n)} - x_{\rho(n)}, f(y_{m, \rho(n)}) - f(x_{\rho(n)})) \to[n] (u_m, v_m)$, so $(u_m, v_m) \in \Whit[x]{f}$.
\end{proof}

\begin{remark}
The Lipschitz continuity of $f$ is needed, as the function $\R^2 \to \R, x \mapsto (x_1)^{1/3}$ shows.
One also needs to fix the point $x \in M$.
In other words, $\Whit{f}$ is not lower hemicontinuous, as the function $\abs{-} \colon \R \to \R$ shows.
\end{remark}

\subsection{Chain rule for Whitney cones}

We begin with a ``tangent analogue'' of Proposition~\ref{prop:sskernels}.

\begin{proposition}\label{prop:Whitney-comp}
Let $M_i$ be manifolds for $i \in \{ 1, 2, 3 \}$.
Let $A_1, B_1 \subseteq M_{12}$ and $A_2, B_2 \subseteq M_{23}$.
Assume that
\begin{enumerate}
\item
$p_{13}$ is injective and proper on $\closure{\opb{p_{12}}(A_1) \cap \opb{p_{23}}(A_2)} \cup \closure{\opb{p_{12}}(B_1) \cap \opb{p_{23}}(B_2)}$,
\item
$(C(A_1, B_1) \times 0_3) \cap (0_1 \times C(A_2, B_2)) \subseteq 0_{123}$.
\end{enumerate}
Then,
\begin{equation}
C(A_1 \comp A_2, B_1 \comp B_2) \subseteq C(A_1, B_1) \comp C(A_2, B_2).
\end{equation}
The reverse inclusion holds if $A_2 = B_2$ is the graph of a $C^1$ map $f$ (without the above two assumptions); if furthermore $M_1 = \{*\}$, it reads $Tf (C(A_1, B_1)) \subseteq C(f(A_1), f(B_1))$.
\end{proposition}

\begin{proof}
Let $(u^1, u^3) \in C_{(x^1, x^3)}(A_1 \comp A_2, B_1 \comp B_2)$.
There exist sequences $(x^1_n, x^3_n) \in A_1 \comp A_2$ and $(y^1_n, y^3_n) \in B_1 \comp B_2$ both converging to $(x^1, x^3)$, and a sequence $c_n > 0$ such that $c_n (y^1_n - x^1_n, y^3_n - x^3_n) \to[n] (u^1, u^3)$.
There exists a sequence $(x^2_n) \in {M_2}^\N$ such that $(x^1_n, x^2_n, x^3_n) \in \opb{p_{12}}(A_1) \cap \opb{p_{23}}(A_2)$, and similarly a sequence $(y^2_n) \in {M_2}^\N$ such that $(y^1_n, y^2_n, y^3_n) \in \opb{p_{12}}(B_1) \cap \opb{p_{23}}(B_2)$.
Since $p_{13}$ is injective and proper on $\closure{\opb{p_{12}}(A_1) \cap \opb{p_{23}}(A_2)} \cup \closure{\opb{p_{12}}(B_1) \cap \opb{p_{23}}(B_2)}$, one can extract subsequences $(x^2_{\rho(n)})$ and $(y^2_{\rho(n)})$ both converging to some $x^2 \in M_2$.

If the sequence $(c_n (y^2_n - x^2_n))$ has a converging subsequence, let $u^2 \in T_{x^2} M_2$ be a limit of a converging subsequence.
Then, $(u^1, u^2) \in C_{(x^1, x^2)}(A_1, B_1)$ and $(u^2, u^3) \in C_{(x^2, x^3)}(A_2, B_2)$.
If not, then $c_n \norm{y^2_n - x^2_n} \to[n] +\infty$ (in some chart of $M_2$ at $x^2$).
Let $d_n \coloneqq \opb{\norm{y^2_n - x^2_n}}$.
We extract a converging subsequence of $(d_n (y^2_n - x^2_n))$ and call $u^2 \in T_{x^2} M_2$ its limit.
One has $\frac{d_n}{c_n} \to[n] 0$, so $d_n(y_n^i  -x_n^i) = \frac{d_n}{c_n} c_n (y_n^i  -x_n^i) \to[n] 0$ for $i \in \{1, 3\}$.
Therefore, the sequences $(x^1_n, x^2_n), (y^1_n, y^2_n)$ and $d_n$ show that $(0, u^2) \in C_{(x^1, x^2)}(A_1, B_1)$ and the sequences $(x^2_n, x^3_n), (y^2_n, y^3_n)$ and $d_n$ show that $(u^2, 0) \in C_{(x^2, x^3)}(A_2, B_2)$ with $u^2 \neq 0$.
This contradicts the second assumption.

\spa
(ii)
Suppose that $A_2 = B_2 = \Gamma_f$ with $f$ a $C^1$ map, so that $C(A_2, B_2) = \Gamma_{Tf}$.
Let $(u^1, u^3) \in C_{(x^1, x^2)}(A_1, B_1) \comp C_{(x^2, x^3)}(A_2, B_2)$.
Therefore, there exists $u^2 \in T_{x^2}M_2$ such that $(u^1, u^2) \in C_{(x^1, x^2)}(A_1, B_1)$ and $T_{x^2}f(u^2) = u^3$.
Therefore, there exist sequences $((x^1_n, x^2_n)) \in {A_1}^\N$ and $((y^1_n, y^2_n)) \in {B_1}^\N$ both converging to $(x^1, x^2)$, and $(c_n) \in (\R_{>0})^\N$ such that $c_n (y^1_n - x^1_n, y^2_n - x^2_n) \to[n] (u^1, u^2)$.
Set $x^3_n \coloneqq f(x^2_n)$ and $y^3_n \coloneqq f(y^2_n)$.
The sequences $((x^1_n, x^3_n))$ and $((y^1_n, y^3_n))$ and $(c_n)$ show that $(u^1, u^3) \in C_{(x^1, x^3)}(A_1 \comp \Gamma_f, B_1 \comp \Gamma_f)$.
\end{proof}

Note that the first assumption of the proposition is satisfied if $A_1 = B_1$ is the graph of a continuous map $f$.
If furthermore $M_3 = \{*\}$, then the second assumption means that $f$ is Lipschitz for $C(A_2, B_2)$ (see definition below) and the conclusion reads $C(\opb{f}(A_2), \opb{f}(B_2)) \subseteq \Whit{f} \circ C(A_2, B_2)$.
The consideration of the second assumption motivates the following definition.

\begin{definition}
Let $f_i \colon M_i \to M_{i+1}$ for $i \in \{ 1, 2 \}$ be continuous maps between manifolds.
The pair $(f_1, f_2)$ is \define{Whitney-regular} if
\begin{equation}
(\Whit{f_1} \times 0_3) \cap (0_1 \times \Whit{f_2}) \subseteq 0_{123}.
\end{equation}
\end{definition}

It will also be convenient to use the following definitions.

\begin{definition}\label{def:Whitney-imm}
Let $f \colon M \to N$ be a continuous map between manifolds and $A \subseteq TM$ and $B \subseteq TN$ be closed cones.
The map $f$ is
\begin{enumerate}
\item
\define{Whitney-immersive} if $\Whit{f} \cap (TM \times 0_N) \subseteq 0_{M \times N}$,
\item
\define{Whitney-immersive for $A$} if $\Whit{f} \cap (A \times 0_N) \subseteq 0_{M \times N}$,
\item
\define{Lipschitz for $B$} if $\Whit{f} \cap (0_M \times B) \subseteq 0_{M \times N}$.
\end{enumerate}
\end{definition}

By Proposition~\ref{prop:LipCone}(1), a continuous map $f \colon M \to N$ is Lipschitz if and only if it is Lipschitz for $TN$.
One can characterize Whitney immersions via a reversed Lipschitz inequality, using arguments similar to the proof of Proposition~\ref{prop:LipCone}(1).
A Whitney immersion is a topological immersion, and a $C^1$-map is a Whitney immersion if and only if it is a $C^1$-immersion.
However, a topological immersion which is smooth need not be a Whitney immersion, as the function $\R \to \R, x \mapsto x^3$ shows.

In view of Proposition~\ref{prop:LipCone}(1), the next proposition is obvious.

\begin{proposition}
Let $f_i \colon M_i \to M_{i+1}$ for $i \in \{ 1, 2 \}$ be continuous maps between manifolds.
If $f_1$ is Lipschitz or $f_2$ is a Whitney immersion, then the pair $(f_1, f_2)$ is Whitney-regular.
\end{proposition}

We can now prove the chain rule for Whitney cones.

\begin{proposition}[chain rule]\label{prop:chainWhit}
Let $f_i \colon M_i \to M_{i+1}$ for $i \in \{ 1, 2 \}$ be continuous maps between manifolds.
If the pair $(f_1, f_2)$ is Whitney-regular, then
\begin{equation}
\Whit{f_2 \circ f_1} \subseteq \Whit{f_1} \circ \Whit{f_2}.
\end{equation}
The reverse inclusion holds if $f_2$ is $C^1$ (without assuming the pair Whitney-regular).
\end{proposition}

\begin{proof}
We apply Proposition~\ref{prop:Whitney-comp} with $A_i= B_i = \Gamma_{f_i}$ for $i \in \{ 1, 2 \}$.
\end{proof}

\begin{example}
Let $f_2 \colon \R \to \R, t \mapsto t^3$ and $f_1 = \opb{f_2}$.
Then, $f_2$ is $C^1$ and $\Whit[0]{f_2 \circ f_1} \supsetneq \Whit[0]{f_1} \circ \Whit[0]{f_2}$.
This shows that the hypothesis that $(f_1, f_2)$ is Whitney-regular is needed.
Similarly, $(f_2, f_1)$ is Whitney-regular and $\Whit[0]{f_1 \circ f_2} \subsetneq \Whit[0]{f_2} \circ \Whit[0]{f_1}$.
This shows that for the reverse inclusion, the hypothesis that $f_2$ is $C^1$ is needed.
\end{example}

\begin{corollary}
The set of Whitney immersions is closed under composition.
\end{corollary}

\begin{proof}
If $f_i \colon M_i \to M_{i+1}$ for $i \in \{ 1, 2 \}$ are Whitney immersions, then $\Whit{f_2 \fcomp f_1} \comp 0_3 \subseteq \Whit{f_1} \circ \Whit{f_2} \circ 0_3 \subseteq \Whit{f_1} \circ 0_2 \subseteq 0_1$, so $f_2 \fcomp f_1$ is a Whitney immersion.
\end{proof}

We end this section with a characterization of Whitney immersions.
By invariance of domain, a Whitney immersion between manifolds of the same dimension is a homeomorphism with Lipschitz inverse.
If the dimensions of the domain and codomain differ, we proceed as follows.
We say that a continuous map $f \colon M \to N$ between manifolds \define{has Lipschitz local retractions} if for any $x \in M$, there are open neighborhoods $U$ of $x$ and $V$ of $f(x)$ such that $f(U) \subseteq V$ and a Lipschitz map $r \colon V \to U$ such that $r \circ f|_U = \id_U$.

\begin{proposition}\label{prop:Whit-imm}
A continuous map between manifolds is a Whitney immersion if and only if it has Lipschitz local retractions.
\end{proposition}

\begin{proof}
Let $f \colon M \to N$ be a Whitney immersion and let $x \in M$.
A Whitney immersion is locally injective, so there is an open neighborhood $U$ of $x$ such that $f$ is injective on $U$.
One can shrink $U$ so that $f(U)$ be included in a chart of $N$ with domain $W$.
The map $\opb{f|_U}$ is defined on $f(U) \subseteq W$, which need not be open, and is not known yet to be continuous.
However, the definition of the Whitney cone and Proposition~\ref{prop:LipCone}(1) still hold in this case.
The cone $\Whit{\opb{f|_U}}$ is the image of $\Whit{f|_U}$ by the ``flip'' $T(U \times W) \to T(W \times U)$.
Therefore, $\Whit[][0]{\opb{f|_U}} \subseteq 0_{W \times U}$, so $\opb{f|_U} \colon f(U) \to U$ is Lipschitz.
Therefore, there exists an open neighborhood $V_1$ of $f(x)$ such that $\opb{f|_U}$ is globally Lipschitz on $f(U) \cap V_1$, and we can assume $f(U) \subseteq V_1$.
Therefore, we can extend $\opb{f|_U}$ to a Lipschitz map $r_1 \colon V_1 \to M$ (see for instance~\cite{Hei}*{\S~2}).
The set $V \coloneqq \opb{r_1}(U)$ is open and contains $f(U)$.
We set $r \coloneqq r_1|_V$.
One has $r \circ f|_U = \id_U$.
The converse is clear, since the relation $r \circ f|_U = \id_U$ implies $\Whit{f|_U} \subseteq \operatorname{flip}(\Whit{r})$ (as seen by applying the definition of the Whitney cone).
\end{proof}

\section{Conormals of maps}
\label{sec:conormal}

In this section, we define the conormal of a continuous map and give its first properties.
We use the convolution of kernels to extend to continuous maps the functorial properties of the microsupport for the four image operations, and we prove a chain rule involving conormals.

\subsection{Definition and first properties}

If $f \colon M \to N$ is a continuous map between manifolds, we set $\Ppg(f) \coloneqq \Ppg(\Gamma_f)$ and
\begin{equation}\label{eq:conormal}
\Lagr{f} \coloneqq \muSupp(\Gamma_f)
\end{equation}
and also $\Ppg_x(f) \coloneqq \Ppg_{(x, f(x))}(\Gamma_f)$ and $\Lagr[x]{f} \coloneqq \muSupp_{(x, f(x))}(\Gamma_f)$ for $x \in M$.
We call $\Lagr{f}$ the \define{conormal of $f$}.
This definition is consistent with Equation~\eqref{eq:lambdaf1}, as Proposition~\ref{prop:subman} shows.

\begin{proposition}
The conormal of $f \colon M \to N$ is a coisotropic closed symmetric cone in $T^*(M \times N)$ satisfying $\pi_{M \times N} (\Lagr{f}) = \Gamma_f$.
\end{proposition}

\begin{proof}
A microsupport is always a coisotropic closed cone.
Since $\Gamma_f$ is closed, one has $\pi_{M \times N} (\Lagr{f}) = \Gamma_f$ by Equation~\eqref{eq:SSproj}.
The graph of a continuous map is a closed $C^0$-submanifold, so its conormal is symmetric by Proposition~\ref{prop:symmetric}.
\end{proof}

\begin{remark}
The conormal of a map need not be a $C^0$-submanifold of half dimension (that is, a $C^0$-Lagrangian), as the following two examples show.
There is a 1-Lipschitz function $f \colon \R \to \R$ such that $\Lagr{f} = \Gamma_f \times \{ (\xi, \eta) \in \R^2 | \abs{\eta} \leq \abs{\xi} \}$.
In particular, $\unpLagr{f}$ is a Lipschitz submanifold with boundary of dimension 3.
There is a continuous function $g \colon \R \to \R$ such that $\Lagr{g} = \Gamma_g \times \R^2$.
In particular, $\Lagr{g}$ is a $C^0$-submanifold of dimension 3.
The idea of the following constructions is taken from a talk by David Preiss.
Let $(U_i)_{i \in \N}$ be a decreasing sequence of open subsets of $\R$ with $\bigcap_i U_i = \Q$ (or any dense subset of $\R$ of measure 0) such that each $U_{n+1}$ has at most half measure in each connected component of $U_n$.
For $x \in \R$, let $\psi(x)$ be the largest index $n$ such that $x \in U_n$ and 0 if $x \notin \bigcup_i U_i$ or $x \in \Q$.
Let $f(x) \coloneqq \int_0^x (-1)^{\psi(x)} \, dx$ and $g(x) \coloneqq \int_0^x (-3/2)^{\psi(x)} \, dx$.
Then, $f$ is 1-Lipschitz, $g$ is continuous, and their conormals are as claimed by the case of equality of the upper bound on the conormal (Theorem~\ref{thm:upper}).
\end{remark}

The following lemma gives a useful criterion for belonging to the propagation set of a sheaf associated with a continuous map.

\begin{lemma}\label{lem:criterFunc}
Let $f \colon M \to N$ be a continuous map between manifolds and let $x \in M$.
Then, $\nu \in \Ppg_x(f)$ if and only if for all $\phi \in C^\infty(W)$, with $W$ an open neighborhood of $(x, f(x))$, such that $\phi(x, f(x)) = 0$ and $d\phi(x, f(x)) = \nu$, the morphism
\begin{equation}
\cor \longrightarrow (\rsect_{\Gamma_f \cap \{ \phi < 0 \}} \cor_{\Gamma_f})_{(x, f(x))}
\end{equation}
induced by the topological embeddings $(p_M \circ i_{W'} \colon \Gamma_f \cap \{ \phi < 0 \} \cap W' \hookrightarrow p_M(W'))_{W' \ni (x, f(x))}$ is an isomorphism.
This implies in particular that the germ of $\Gamma_f \cap \{ \phi < 0 \}$ at $(x, f(x))$ has the cohomology of a point.
\end{lemma}

\begin{proof}
We apply Lemma~\ref{lem:criterSet} with $Z = \Gamma_f$.
The results follows since the isomorphism $p_M|_{\Gamma_f} \colon \Gamma_f \to M$ induces an isomorphism $\cor = (\cor_M)_x \isoto (\cor_{\Gamma_f})_{(x, f(x))}$.
\end{proof}

\begin{remark}
In most of the proofs below, we only use the fact that having the cohomology of a point implies being nonempty.
So in a sense, we use only crude estimates, which shows the efficiency of sheaf theoretical methods.
\end{remark}

We will need the following result in later sections.
It shows that adding a strictly differentiable map to a continuous map shears its tangent cone and its conormal.

\begin{lemma}[shearing lemma]\label{lem:coneShear}
Let $f, g \colon M \to \R^n$ be continuous maps on a manifold with $g$ strictly differentiable at $x \in M$.
Then,
\begin{align}
\Whit[x]{f + g} &= (p_M, p_N + g'(x) \fcomp p_M) (\Whit[x]{f})
\qquad\text{and}\\
\Lagr[x]{f + g} &= (p_M - g'(x)^\intercal \fcomp p_N, p_N) (\Lagr[x]{f}).
\end{align}
\end{lemma}

\begin{proof}
(i)
The case of the Whitney cone is straightforward (directly or using Proposition~\ref{prop:whitney}(8)).

\spa
(ii)
Fix a chart of $M$ at $x$.
We introduce the topological automorphism $\Phi \colon \R^{m+n} \to \R^{m+n}, (x, y) \mapsto \Phi(x, y) = (x, y - g(x))$.
Then, $\Phi(\Gamma_f) = \Gamma_{f+g}$.
Since $\Phi$ is strictly differentiable on $\{x\} \times \R^n$ as well as its inverse, it follows from Remark~\ref{rem:funcSSstrictDiff} that $\Lagr[x]{f} = \Lagr[(x, f(x))]{\Phi} \acomp \Lagr[x]{f+g}$ and a simple computation completes the proof.
\end{proof}

\subsection{Chain rule for conormals}

In this subsection, we will apply Proposition~\ref{prop:sskernels} to kernels associated with maps, in order to derive upper bounds on microsupports of direct and inverse images of sheaves, and then on the conormal of a composite map.

By analogy with the smooth case (\textit{cf.} Equations~\eqref{eq:subm} and~\eqref{eq:non-char}), we make the following definitions (compare with Definition~\ref{def:Whitney-imm}).

\begin{definition}\label{def:muSubm}
Let $f \colon M \to N$ be a continuous map between manifolds, and $A \subseteq T^*M$ and $B \subseteq T^*N$ be closed cones.
The map $f$ is
\begin{enumerate}
\item
a \define{microlocal submersion} if $\Lagr{f} \cap (0^*_M \times T^*N) \subseteq 0^*_{M \times N}$,
\item
\define{non-characteristic for $B$} if $\Lagr{f} \cap (0^*_M \times B)\subseteq 0^*_{M \times N}$,
\item
\define{Lipschitz for $A$} if $\Lagr{f} \cap (A \times 0^*_N) \subseteq 0^*_{M \times N}$.
\end{enumerate}
\end{definition}

A $C^1$ map is obviously Lipschitz for $T^*M$.
We will see in Theorem~\ref{thm:Lip} that ``Lipschitz for $T^*M$'' is equivalent to ``Lipschitz''.

\begin{example}
For $f \colon \R^2 \to \R, x \mapsto (x_1)^{1/3}$, one has $\Lagr[x]{f} = \R (-1, 0, 3 \abs{x_1}^{2/3})$, so $f$ is Lipschitz for the constant cone $\R^2 \times \{ (\xi_1, \xi_2) \in \R^2 \mid \abs{\xi_1} \leq C \abs{\xi_2} \}$ for any $C \in \R$.
\end{example}

We first prove the following generalization of Proposition~\ref{prop:funcSS}(2) to continuous maps.
A continuous map is said to be Lipschitz (resp.\ non-characteristic) for a sheaf if it is Lipschitz (resp.\ non-characteristic) for its microsupport.

\begin{proposition}\label{prop:funcSSCont}
Let $f \colon M \to N$ be a continuous map between manifolds.
\begin{enumerate}
\item
Let $F \in \Derb(\cor_M)$ and assume that $f$ is Lipschitz for $F$ and proper on $\supp(F)$.
Then, $\muSupp(\roim{f}F) \subseteq \muSupp(F) \acomp \Lagr{f}$.
\item
Let $G \in \Derb(\cor_{N})$ and assume that $f$ is non-characteristic for $G$.
Then, the morphism $\opb{f} G \Ltens \omega_{M/N} \to \epb{f} G$ is an isomorphism and $\muSupp(\opb{f} G) \subseteq \Lagr{f} \acomp \muSupp(G)$.
\end{enumerate}
\end{proposition}

\begin{proof}\footnote{This proof was obtained jointly with Pierre Schapira.}
The bounds on the microsupport are straightforward applications of Proposition~\ref{prop:sskernels} (using Proposition~\ref{prop:kerFunc}).
As for the morphism $\opb{f} G \Ltens \omega_{M/N} \to \epb{f} G$, we first note that it is functorial in $f$.
Therefore, the lemma following this proof shows that it is enough to prove that it is an isomorphism when $f$ is a closed embedding and when $f$ is a submersion.
Indeed, we use the decomposition $f = p_N \circ (\id_M, f)$, that is, the inclusion in the graph of $f$ (which is a closed embedding) followed by the projection on the second factor (which is a smooth submersion).
The submersion case is treated in~\cite{KS}*{Prop.~3.3.2}.

Suppose that $f$ is a closed embedding of $M$ into $N$.
For a conic sheaf $F \in \Derb(\cor_{T^*N})$, one has the Sato distinguished triangle $\reim{\pi} F \to \roim{\pi} F \to \roim{\unp\pi} F \to[+1]$.
When $F = \muhom(\cor_M, G)$ (with $\cor_M = \cor_{N, M} \in \Derb(\cor_N)$), one has $\reim{\pi} \muhom(\cor_M, G) = \rhom(\cor_M, \cor_N) \Ltens G$ and $\roim{\pi} \muhom(\cor_M, G) = \rhom(\cor_M, G) = \rsect_M G$ (by~\cite{KS}*{Prop.~4.4.2(i)}).
Non-characteristicity says that $\muSupp(G) \cap \muSupp(\cor_M) \subseteq 0^*_N$.
By~\cite{KS}*{Cor.~5.4.10(ii)}, one has $\supp(\muhom(\cor_M, G)) \subseteq \muSupp(G) \cap \muSupp(\cor_M) \subseteq 0^*_N$, so $\roim{\unp\pi} \muhom(\cor_M, G) = 0$.
Therefore, $\rsect_M G \simeq \rhom(\cor_M, \cor_N) \Ltens G$.

Since $f$ is a closed embedding, one has $\epb{f}G \simeq \opb{f} (\rsect_M G)$ by~\cite{KS}*{Prop.~3.1.12}.
Applying this formula to $\cor_N$, we also have $\omega_{M/N} = \epb{f}\cor_N \simeq \opb{f} (\rhom(\cor_M, \cor_N))$.
Therefore, $\epb{f} G \simeq \opb{f} (\rsect_M G) \simeq \opb{f} (\rhom(\cor_M, \cor_N) \Ltens G) \simeq \omega_{M/N} \Ltens \opb{f} G$.
\end{proof}

\begin{lemma}
Let $f \colon M \to N$ be a continuous map between manifolds.
Let $s \colon T^*M \times_M T^*M \to T^*M$ be the fiberwise subtraction.
Then, $\Lagr{(\id_M, f)} \subseteq \opb{(s \times \id_{T^*N})}(\Lagr{f})$.
In particular, if $f$ is non-characteristic for $G \in \Derb(\cor_N)$, then the closed embedding $(\id_M, f)$ is non-characteristic for $\opb{p_N}(G)$.
\end{lemma}

\begin{proof}
We set $\hat{f} \coloneqq (\id_M, f)$ and $P \coloneqq \Delta_M \times N$.
The latter is a closed submanifold of $M \times M \times N$.
One has $\Gamma_{\hat{f}} = P \cap \opb{p_{MN}}(\Gamma_f)$, so $K_{\hat{f}} = (\opb{p_{MN}} K_f)_P$.
One has $\muSupp(\opb{p_{MN}} K_f) = \Lagr{p_{MN}} \acomp \Lagr{f} = 0^*_M \times \Lagr{f}$ and $T^*_P (M \times M \times N) = \{ (\xi, -\xi, 0) \mid \xi \in T^*M \}$.
Therefore, $\muSupp(\opb{p_{MN}} K_f) \cap T^*_P (M \times M \times N) \subseteq 0^*_{MMN}$, so by~\cite{KS}*{Cor.~5.4.11(i)}, one has $\Lagr{\hat{f}} \subseteq \muSupp(\opb{p_{MN}} K_f) + T^*_P (M \times M \times N) = \opb{(s \times \id_{T^*N})}(\Lagr{f})$.

As for the second claim, let $G \in \Derb(\cor_N)$.
Then, $\muSupp(\opb{p_N}G) = \Lagr{p_N} \acomp \muSupp(G) = 0^*_M \times \muSupp(G)$.
Therefore, $\Lagr{\hat{f}} \cap (0^*_M \times \muSupp(\opb{p_N}G)) \subseteq \opb{(s \times \id_{T^*N})}(\Lagr{f}) \cap (0^*_{MM} \times \muSupp(G)) \subseteq 0^*_M \times (\Lagr{f} \cap (0^*_M \times \muSupp(G))$.
\end{proof}

As an application, let $f \colon M \to N$ be a closed topological embedding between manifolds.
If $f$ is a closed $C^1$-embedding, we know that $T^*_{f(M)} N = 0_M^* \circ \Lambda_f$.
In the continuous case, we have $\roim{f} \cor_M \simeq \cor_{f(M)}$, so the proposition tells us that $\muSupp(f(M)) \subseteq 0_M^* \circ \Lambda_f$.

\quad

Now, we will apply Proposition~\ref{prop:sskernels} when both kernels are associated with maps, in order to derive an upper bound on the conormal of a composite map.
This constitutes an analog for continuous maps of the usual chain rule for differentiable maps.
The hypotheses of Proposition~\ref{prop:sskernels} applied to $K_i \coloneqq K_{f_i}$ (notation introduced in Equation~\eqref{eq:sheafFunc}) with $f_i \colon M_i \to M_{i+1}$ continuous maps between manifolds for $i \in \{ 1, 2 \}$, read
\begin{enumerate}
\item
$p_{13}$ is proper on $\opb{p_{12}}(\Gamma_{f_1}) \cap \opb{p_{23}}(\Gamma_{f_2})$,
\item
$(\Lagr{f_1} \times 0^*_3) \cap (0^*_1 \times \Lagr{f_2}) \subseteq 0^*_{123}$
\end{enumerate}
(one can remove the antipodal map from the second condition since conormals of maps are symmetric).
Since $f_1$ is continuous, the first hypothesis is satisfied.
This motivates the following definition.

\begin{definition}\label{def:regular}
Let $f_i \colon M_i \to M_{i+1}$ for $i \in \{ 1, 2 \}$ be continuous maps between manifolds.
The pair $(f_1, f_2)$ is \define{regular} if
\begin{equation}
(\Lagr{f_1} \times 0^*_3) \cap (0^*_1 \times \Lagr{f_2}) \subseteq 0^*_{123}.
\end{equation}
\end{definition}

It is convenient to define
\begin{equation}\label{eq:directionalConormal}
\Lagr[x][\eta]{f} \coloneqq \Lagr[x]{f} \cap \big( T^*_x M \times \R_{\geq 0} \eta \big)
\end{equation}
for $(x, \eta) \in M \times_N T^*N$ and
\begin{equation}
\Lagr[][0]{f} \coloneqq \Lagr{f} \cap \big( T^*M \times 0^*_N \big).
\end{equation}
Therefore, a map is Lipschitz for $T^*M$ if and only if
\begin{equation}\label{eq:regular}
\Lagr[][0]{f} \subseteq 0^*_{M \times N}.
\end{equation}

\begin{example}
The function $f \colon \R \to \R, t \mapsto t^3$ is a topological automorphism which is smooth but is not microlocally submersive and its inverse is not Lipschitz for $T^*\R$, and the pair $(f, \opb{f})$ is not regular.
In particular, a topological submersion need not be a microlocal submersion, even if it is smooth.
\end{example}

The next proposition is obvious.

\begin{proposition}
Let $f_i \colon M_i \to M_{i+1}$ for $i \in \{ 1, 2 \}$ be continuous maps between manifolds.
If $f_1$ is microlocally submersive or if $f_2$ is Lipschitz for $T^*M_3$, then the pair $(f_1, f_2)$ is regular.
\end{proposition}

We can now prove the chain rule for conormals.

\begin{proposition}[chain rule]\label{prop:chain}
Let $f_i \colon M_i \to M_{i+1}$ for $i \in \{ 1, 2 \}$ be continuous maps between manifolds.
If the pair $(f_1, f_2)$ is regular, then
\begin{equation}\label{eq:chain}
\Lagr{f_2 \fcomp f_1} \subseteq \Lagr{f_1} \acomp \Lagr{f_2}
\end{equation}
with equality if $f_1$ is a $C^1$-submersion (and if $f_1$ and $f_2$ are both $C^1$).
\end{proposition}

\begin{proof}
As we have seen above, both hypotheses of Proposition~\ref{prop:sskernels} are satisfied when the pair $(f_1, f_2)$ is regular, so the inclusion follows from Proposition~\ref{prop:compFunc}.

If $f_1$ is a $C^1$-submersion, then $\tilde{f}_1 \coloneqq f_1 \times \id_3 \colon M_{13} \to M_{23}$ is a $C^1$-submersion and $\Gamma_{f_2 \circ f_1} = \opb{\tilde{f}_1}(\Gamma_{f_2})$.
By Equation~\eqref{eq:opbSet} (case of equality), one has $\Lagr{f_2 \circ f_1} = \muSupp(\Gamma_{f_2 \circ f_1}) = \muSupp(\opb{\tilde{f}_1}(\Gamma_{f_2})) = \Lagr{\tilde{f}_1} \acomp[23] \muSupp(\Gamma_{f_2}) = \Lagr{\tilde{f}_1} \acomp[23] \Lagr{f_2}$.
The result follows since for any $B \subseteq T^*M_{23}$, one has $\Lagr{\tilde{f}_1} \acomp[23] B = \Lagr{f_1} \acomp[2] B$ (since $f_1$ is $C^1$).
\end{proof}

\begin{example}
The inclusion in the chain rule may be strict: take $f_2 \colon \R \to \R, t \mapsto t^3$ and $f_1 = \opb{f_2}$.

If the pair $(f_1, f_2)$ is not regular, then the chain rule need not hold.
For example, take $f_1 \colon \R \to \R, t \mapsto t^3$ and $f_2 = \opb{f_1}$.
\end{example}

\begin{corollary}\label{cor:regular1}
The set of microlocal submersions is closed under composition.
\end{corollary}

\begin{proof}
If $f_i \colon M_i \to M_{i+1}$ for $i \in \{ 1, 2 \}$ are microlocal submersions, then $0^*_1 \circ \Lagr{f_2 \fcomp f_1} \subseteq 0^*_1 \circ \Lagr{f_1} \acomp \Lagr{f_2} \subseteq 0^*_2 \circ \Lagr{f_2} \subseteq 0^*_3$, so $f_2 \fcomp f_1$ is a microlocal submersion.
\end{proof}

We end this section with a result giving sufficient conditions for the direct and inverse images of closed cones in a (co)tangent bundle to be closed.

\begin{proposition}\label{prop:closed-cone}
Let $f \colon M \to N$ be a continuous map between manifolds.
\begin{enumerate}
\item
If $f$ is Lipschitz for the closed cone $B \subseteq TN$, then $\Whit{f} \circ B$ is closed in $TM$.
\item
If $f$ is non-characteristic for the closed cone $B \subseteq T^*N$, then $\Lagr{f} \circ B$ is closed in $T^*M$.
\item
If $f$ is Whitney-immersive for the closed cone $A \subseteq TM$ and proper on $\tau_M(A)$, then $A \circ \Whit{f}$ is closed in $TN$.
\item
If $f$ is Lipschitz for the closed cone $A \subseteq T^*M$ and proper on $\pi_M(A)$, then $A \circ \Lagr{f}$ is closed in $T^*N$.
\end{enumerate}
\end{proposition}

In the smooth case, the proof of Item~2 is sketched in the paragraph following the definition of non-characteristic morphisms in~\cite{KS}*{Def.~5.4.12}.

\begin{proof}
The four claims have similar proofs, so we only give a proof of the second one.
Let $K \subseteq M \times N$ be compact and contained in a chart.
Since $\Lagr{f} \cap (0^*_M \times B) \subseteq 0^*_{M \times N}$ and $\Lagr{f}$ is closed, there exists $\alpha_K > 0$ such that $\opb{\pi_{M \times N}}(K) \cap \Lagr{f} \cap (T^*M \times B) \subseteq \{ (\xi, \eta) \in T^*(M \times N) \mid \norm{\xi} \geq \alpha_K \norm{\eta} \}$.
Therefore, $p_M$ is proper on $\opb{\pi_{M \times N}}(K) \cap \Lagr{f} \cap (T^*M \times B)$.
This implies that $p_M(\opb{\pi_{M \times N}}(K)) \cap (\Lagr{f} \comp B)$ is closed.
The requirement that $K$ be contained in a chart can be dropped.
Note that $p_M \circ \opb{\pi_{M \times N}} = \opb{\pi_M} \circ p_M$.

Let $x \in M$.
Let $K_N$ be a compact neighborhood of $f(x)$.
There exists a compact neighborhood $K_M$ of $x$ included in $\opb{f}(K_N)$.
Set $U \coloneqq p_M(\opb{\pi_{M \times N}}(K_M \times K_N)) = \opb{\pi_M}(K_M)$.
Then, $U \cap (\Lagr{f} \comp B)$ is closed.
Therefore, any point $(x, \xi) \in T^*M$ has a neighborhood $U$ such that $U \cap (\Lagr{f} \comp B)$ is closed.
This implies that $\Lagr{f} \comp B$ is closed.
\end{proof}

\begin{remark}
With the notation of Definition~\ref{def:regular}, if $x_1 \in M_1$, one says that the pair $(f_1, f_2)$ is regular at $x_1$ if $(\Lagr[x_1]{f_1} \times \{0\}) \cap (\{0\} \times \Lagr[f(x_1)]{f_2}) = \{0\}$.
If $(f_1, f_2)$ is regular at $x_1$, then it is so in a neighborhood of $x_1$, since conormals are closed cones.
The same remark applies for the notions of Whitney-regular pair and of microlocally submersive, non-characteristic, Whitney immersive and Lipschitz maps.
\end{remark}

As for a characterization of microlocal submersions, in view of Proposition~\ref{prop:Whit-imm}, a reasonable conjecture is that the microlocal submersions are the continuous maps which have Lipschitz local sections.
Here is a result in that direction.\footnote{Pierre Schapira showed me how an adaptation of the proof of the microlocal Bertini--Sard theorem~\cite{KS}*{Prop.~8.3.12} shows that subanalytic microlocal submersions are open maps.}

\begin{proposition}\label{prop:open}
A microlocal submersion with 1-dimensional codomain is an open map.
\end{proposition}

\begin{proof}
Let $f \colon M \to \R$ be a microlocal submersion and let $x_0 \in M$.
We can suppose that $M$ is open in $\R^m$.
Define $\phi \colon M \times \R \to \R, (x, y) \mapsto y - f(x_0)$.
Then $\phi(x_0, f(x_0)) = 0$ and $d\phi(x_0, f(x_0)) = (0, 1) \in \Ppg_{x_0}(f)$.
Therefore, by Lemma~\ref{lem:criterFunc}, for any neighborhood $U$ of $(x_0, f(x_0))$, one has $U \cap \Gamma_f \cap \{ \phi < 0 \} \neq \varnothing$.
This means that arbitrarily close to $x_0$, there are points $x$ with $f(x) < f(x_0)$, and similarly, points $y$ with $f(y) > f(x_0)$.
Since $M$ is locally connected, this implies that $f$ is open.
\end{proof}

\section{Real-valued functions}
\label{sec:real-valued}

In this section, we study more precisely the case of real-valued functions.
We introduce directional Dini derivatives, which permit to give precise descriptions of the Whitney cones related to the graph and epigraph of a function.
Then, we study extrema of real-valued functions and prove an analogue of Fermat's lemma for continuous functions.
In the third subsection, we relate the conormal of a continuous function to the microsupport of the constant sheaf on its epigraph.

\subsection{Directional Dini derivatives}

If $\V$ is a vector space, $f \colon \V \to \R$ and $(x, u) \in T\V$, we define the \define{supremal derivative} and \define{supremal quotient} of $f$ at $(x, u)$ respectively by
\begin{align}
\ovD f(x, u) &\coloneqq \limsup_{\subalign{t &\to 0^+\\v &\to u}} \frac{f(x + t v) - f(x)}{t}
\qquad\text{and}\\
\label{eq:sup-quotient}
\ovQ f(x, u) &\coloneqq \limsup_{\mathclap{\subalign{t &\to 0^+\\(y, v) &\to (x, u)}}} \frac{f(y + t v) - f(y)}{t}.
\end{align}

One has $\ovQ f = \limsup \ovD f \colon T\V \to \closure{\R}$.
The functions $\ovD f(x, -)$ and $\ovQ f(x, -)$ are $\R_{> 0}$-homogeneous on $T_x \V$ for any $x \in \V$.
The functions $\unD f$ and $\unQ f$ are defined similarly with $\liminf$.
One has $\ovQ f(x, -u) = - \unQ f(x, u)$ for any $(x, u) \in T\V$.
If $\V = \R$, we recover the usual Dini derivatives: for instance, $\ovD f(x, 1) = D^+ f(x)$ and $\ovD f(x, -1) = -D_- f(x)$.

Once a norm is fixed in $\V$, one has (with the notation introduced in Definition~\ref{def:lip}; the maxima are in $\closure{\R}$)
\begin{align}
\Lippw f(x) &= \max_{\substack{u \in \V\\ \norm{u} = 1}} \left( \abs{\overline{D} f(x, u)}, \abs{\underline{D} f(x, u)} \right)
\qquad\text{and}\\
\Lip f(x) &= \max_{\substack{u \in \V\\ \norm{u} = 1}} \: \abs*{\overline{Q} f(x, u)}.
\end{align}

The following proposition, which relates the Whitney cone to the directional Dini derivatives, will be needed in the proof of the upper bound on the conormal of a map.

\begin{proposition}\label{prop:Dini}
Let $\V$ be a vector space and $f \colon \V \to \R$.
For any $x \in \V$, one has
\begin{align}
C_x(\Gamma_f) &= \left\{ (u, t) \in T_x \V \times \R \mathrel{\Big|} \unD f(x, u) \leq t \leq \ovD f(x, u) \right\},\\
\Whit[x]{f} &= \left\{ (u, t) \in T_x \V \times \R \mathrel{\Big|} \unQ f(x, u) \leq t \leq \ovQ f(x, u) \right\}
\quad\text{and}\\
N_x(\Gamma^+_f) &= \left\{ (u, t) \in T_x \V \times \R \mathrel{\Big|} t > \ovQ f(x, u) \right\}.
\end{align}
\end{proposition}

\begin{proof}
We treat the case of extremal quotients, the case of extremal derivatives being similar.

\spa
(i)
Inclusion ``$\subseteq$''.
Let $(u, t) \in \Whit[x]{f}$.
There are sequences $x_n, y_n \to[n] x$ and $c_n > 0$ such that $c_n (y_n - x_n, f(y_n) - f(x_n)) \to[n] (u, t)$.
Since $y_n = x_n + \frac{1}{c_n} (c_n (y_n - x_n))$, one has $\ovQ f(x, u) \geq \lim_n c_n (f(y_n) - f(x_n)) = t$.
One proves similarly that $\unQ f(x, u) \leq t$.

\spa
(ii)
Inclusion ``$\supseteq$''.
If $u \in T_x \V$ and $t \in \intervC{\unQ f(x, u), \ovQ f(x, u)} \cap \R$, then by definition of the $\liminf$ and $\limsup$, for any $\epsilon > 0$, there exist sequences $y_n, z_n \to[n] x$ and $v_n, w_n \to[n] u$ and $a_n, b_n \to[n] 0^+$ such that
\begin{equation}
\frac{f(y_n + a_n v_n) - f(y_n)}{a_n} - \epsilon \leq t \leq \frac{f(z_n + b_n w_n) - f(z_n)}{b_n} + \epsilon
\end{equation}
for all $n \in \N$.
By the intermediate value theorem, there exists $t_n \in [0, 1]$ such that for $(x_n, d_n, u_n) \coloneqq (1-t_n) (y_n, a_n, v_n) + t_n (z_n, b_n, w_n)$, one has
\begin{equation}
\abs*{\frac{f(x_n + d_n u_n) - f(x_n)}{d_n} - t} \leq \epsilon
\end{equation}
for all $n \in \N$.
Then, with $c_n \coloneqq \opb{d_n}$ and $s_n \coloneqq x_n + d_n u_n$, one has
$x_n, s_n \to[n] x$ and $c_n > 0$ and $c_n (s_n - x_n) = u_n \to[n] u$ and $\limsup_n \abs{c_n (f(s_n) - f(x_n)) - t} \leq \epsilon$.
In particular, there exists $a \in [-\epsilon, \epsilon]$ such that $(u, t + a) \in \Whit[x]{f}$.
Since this is true for any $\epsilon > 0$, the result follows from the closedness of $\Whit[x]{f}$.

\spa
(iii)
The proof for $N_x(\Gamma^+_f)$ is similar.
\end{proof}

\begin{remark}
In particular, if $f$ is Lipschitz at $x$, then $\ovQ f(x, 0) = \unQ f(x, 0) = 0$, else $\ovQ f(x, 0) = - \unQ f(x, 0) = +\infty$.
\end{remark}

\begin{corollary}
Let $f \colon M \to \R$ be a function on a manifold.
For any $x \in M$, one has a partition
\begin{equation}
\Whit[x]{f} \sqcup N_x(\Gamma^+_f) \sqcup N_x(\Gamma^-_f) = T_x M \times \R.
\end{equation}
\end{corollary}

\begin{proof}
It follows easily from the proposition and Proposition~\ref{prop:strict-cone}.
\end{proof}

\subsection{First-order extrema}

Since Whitney cones and microsupports depend only on the $C^1$-structure of a manifold, it is natural to introduce the notion of first-order extremum.
If $\V$ is a normed vector space, $f \colon \V \to \R$ and $x \in \V$, we define
\begin{align}
\und f(x) \coloneqq & \liminf_{v \to 0} \frac{f(x + v) - f(x)}{\norm{v}},\label{eq:unD}\\
\ovd f(x) \coloneqq & \limsup_{v \to 0} \frac{f(x + v) - f(x)}{\norm{v}}.\label{eq:ovD}
\end{align}

Let $M$ be a manifold.
If $f \colon M \to \R$ and $x \in M$ and $\norm{-}$ is a norm on $T_xM$, then $\liminf_{v \to 0} \frac1{\norm{v}} \big( f(\opb{\phi}(\phi(x) + \phi'(x)v)) - f(x) \big)$ does not depend on the chart $\phi$ at~$x$.
Therefore, when no norm is specified, the extended reals $\und f(x)$ and $\ovd f(x)$ are well-defined up to multiplication by a strictly positive real number.
Therefore, properties like ``$\und f(x) > 0 $'' still make sense.

Using implicitly the canonical identifications $T(M \times \R) \simeq TM \times T\R$ and $T\R \simeq \R \times \R$, we define
\begin{equation}
T_{\geq 0}(M \times \R) \coloneqq TM \times (\R \times \R_{\geq 0})
\end{equation}
and similarly for the cotangent bundle, for $T_{=0}(M \times \R)$, etc.

\begin{definition}[First-order extremum]
Let $f \colon M \to \R$ be a function on a manifold.
A point $x \in M$ is a \define{first-order minimum} (or f-o minimum) of $f$ if $\und f(x) \geq 0$.
It is an f-o maximum if $\ovd f(x) \leq 0$, and an f-o extremum if it is either an f-o minimum or an f-o maximum.
\end{definition}

A local minimum (resp.\ maximum, extremum) is obviously an f-o minimum (resp.\ maximum, extremum).
A point which is both an f-o minimum and an f-o maximum is a \define{stationary point}: the function is differentiable at that point with derivative zero.
We have the following characterization of f-o extrema.

\begin{proposition}\label{prop:f-o}
Let $M$ be a manifold, $x \in M$, and $f \colon M \to \R$ be a function.
Then, the following are equivalent:
\begin{enumerate}
\item
$x$ is an f-o minimum of $f$,
\item
$C_x (\Gamma_f) \subseteq T_{\geq 0}(M \times \R)$,
\item
$(0, 1) \in C_x (\Gamma_f)^\circ$,
\item
there exist an open neighborhood $U$ of $x$ and a function $\psi \in C^1(U)$ such that $\psi(x) = f(x)$ and $d\psi(x) = 0$ and $\psi \leq f|_U$.
\end{enumerate}
\end{proposition}

\begin{proof}
(4)$\Rightarrow$(1)$\Rightarrow$(2)$\Rightarrow$(3) are obvious.

\spa
(3)$\Rightarrow$(4).
Since $(0, 1) \in C_x (\Gamma_f)^\circ$, Lemma~\ref{lem:cone} gives an open neighborhood $W$ of $(x, f(x))$ and $\phi \in C^1(W)$ with $\phi(x, f(x)) = 0$ and $d\phi(x, f(x)) = (0, 1)$ and $\phi(\Gamma_f \cap W) \subseteq \R_{\geq 0}$.
The implicit function theorem gives an open neighborhood $U \times V \subseteq W$ of $(x, f(x))$ and $\psi \in C^1(U)$ with $\psi(x) = 0$ and $d\psi(x) = 0$ such that for $(x, y) \in U \times V$, $\phi(x, y) \geq 0$ is equivalent to $y \geq \psi(x)$.
Therefore, $\psi \leq f|_U$.
\end{proof}

We obtain the following corollary of independent interest.

\begin{corollary}
Let $M$ be a manifold and $\W$ be a normed vector space.
Let $x \in M$ and $f \colon M \to \W$ be a function differentiable at~$x$ with $df(x) = 0$.
Then, there exist an open neighborhood $U$ of $x$ and a function $\psi \in C^1(U)$ with $\psi(x) = 0$ and $d\psi(x) = 0$ such that $\norm{f|_U - f(x)}_\W \leq \psi$ (that is, for any $y \in U$, one has $\norm{f(y) - f(x)}_\W \leq \psi(y)$).
\end{corollary}

\begin{proof}
The hypothesis $df(x) = 0$ implies that $x$ is an f-o maximum of $\norm{f - f(x)}_\W$ and the result follows from Proposition~\ref{prop:f-o}.
\end{proof}

\begin{remark}
One cannot strengthen the conclusion of this corollary nor of Item~4 of Proposition~\ref{prop:f-o} to $\psi \in C^2(U)$, as the example of Remark~\ref{rmk:cone} shows.
\end{remark}

As in standard calculus, the following Fermat lemma will be used to prove Rolle's lemma and the mean value theorem for continuous functions.

\begin{proposition}[Fermat lemma]\label{prop:fermat}
Let $M$ be a manifold and $f \colon M \to \R$ be a continuous function.
If $f$ has an f-o extremum at $x \in M$, then
\begin{align}
T_x M \times \{0\} &\subseteq \Whit[x]{f}
\qquad\text{and}\\
\fantome[T_x M]{c}{\{0\}} \times \fantome[\{0\}]{c}{\R} &\subseteq \Lagr[x]{f}.
\end{align}
\end{proposition}

\begin{proof}
(i)
Whitney cone.
We can suppose that $x$ is an f-o minimum.
First, suppose that $M = \R$ and $x = 0$.
Let $\epsilon > 0$ and set $f_\epsilon (t) \coloneqq f(t) + \epsilon \abs{t}$.
Then, $0$ is a local strict minimum of $f_\epsilon$.
We can suppose that it is a global strict minimum.

If $f_\epsilon(-1) \leq f_\epsilon(1)$, then set $x_0 \coloneqq -1$ and $y_0 \coloneqq \min \{ z \in \intervOC{0, 1} \mid f_\epsilon(z) = f_\epsilon(x_0) \}$ (which exists by the intermediate value theorem and is a minimum by continuity of $f_\epsilon$).
If $f_\epsilon(-1) > f_\epsilon(1)$, then set $y_0 \coloneqq 1$ and $x_0 \coloneqq \max \{ z \in \intervCO{-1, 0} \mid f_\epsilon(z) = f_\epsilon(y_0) \}$.
Use the same procedure with the points $\pm 1$ replaced with the points $\pm \min (-x_0, y_0)/2$, to construct $(x_1, y_1)$.
This way, one constructs sequences $x_n, y_n \to[n] 0$ that show that $(1, 0) \in \Whit[0]{f_\epsilon}$, so there exists $\alpha_\epsilon \in \intervC{-\epsilon,\epsilon}$ such that $(1,\alpha_\epsilon) \in \Whit[0]{f}$.
Since this is true for any $\epsilon > 0$, the closedness of $\Whit[0]{f}$ implies that $\R \times \{0\} \subseteq \Whit[0]{f}$.

In the general case, suppose that $M$ is open in $\R^m$.
Let $u \in \unp{T}_xM$ and set $\gamma \colon \intervO{-\alpha, \alpha} \to M, t \mapsto x+ t u$.
By the previous paragraph, $\R \times \{0\} \subseteq \Whit[0]{f \circ \gamma}$, and by the chain rule for Whitney cones ($\gamma$ being Lipschitz), one has $\Whit[0]{f \circ \gamma} \subseteq \Whit[0]{\gamma} \circ \Whit[x]{f}$.
But $\Whit[0]{\gamma} = \R (1, u)$, so $(u, 0) \in \Whit[x]{f}$.

\spa
(ii)
Conormal.
If $x$ is an f-o minimum of $f$, then $(0, 1) \in C_x (\Gamma_f)^\circ$ by Proposition~\ref{prop:f-o}, so the result follows from Proposition~\ref{prop:SSbounds}.
If $x$ is an f-o maximum of $f$, then it is an f-o minimum of $-f$, so by the above, one has $\{0\} \times \R \subseteq \Lagr[x]{-f}$.
Then, the result follows from the fact that $\Lagr{f}$ and $\Lagr{-f}$ are mapped onto each other by the involutive automorphism of $T^*(M \times \R)$ given by $(x, t; u, \tau) \mapsto (x, -t; u, -\tau)$.
\end{proof}

\subsection{Microsupports of epigraphs}

In~\cite{Vic}, N.~Vichery studied the microsupport of the constant sheaf on the epigraph of a real-valued function, rather than on its graph.
In this subsection, we show that the two points of view are equivalent for Lipschitz functions.

Let $M$ be a manifold and let $f \colon M \to \R$ be a continuous function.
We introduce the microsupports
\begin{equation}
\Lambda^\pm_f \coloneqq \muSupp(\Gamma^\pm_f).
\end{equation}

One has the exact sequences (by~\cite{KS}*{Prop.~2.3.6(v-vi)})
\begin{align}
\label{eq:exact1}
0 \to \cor_{M \times \R} \to \cor_{\Gamma^+_f} \oplus \cor_{\Gamma^-_f} \to \cor_{\Gamma_f} \to 0
&\qquad\text{and}\\
\label{eq:exact2}
0 \to \cor_{\interior{\Gamma^\pm_f}} \to \cor_{M \times \R} \to \cor_{\Gamma^\mp_f} \to 0.
\end{align}

It follows from~\cite{KS}*{Exe.~III.4} that if $U$ is a convex open subset of a vector space, then $\cor_U$ is cohomologically constructible and
\begin{equation}
\RD_M'(\cor_U) \simeq \cor_{\closure{U}}
\end{equation}
so by Proposition~\ref{prop:dual}, one has
\begin{equation}\label{eq:SSanti}
\muSupp(\closure{U}) = \muSupp(U)^a.
\end{equation}
Note that these properties are of a topological and local nature.

\begin{proposition}\label{prop:epigraph}
Let $M$ be a manifold and let $f \colon M \to \R$ be a continuous function.
Then,
\begin{enumerate}
\item
$\unpLagr{f} = \unpLagr{f}^+ \cup \unpLagr{f}^-$,
\item
$\unpLagr{f}^- = (\unpLagr{f}^+)^a$,
\item
$\Lambda^+_f \subseteq T^*_{\geq 0}(M \times \R)$.
\end{enumerate}
\end{proposition}

This proposition implies that the knowledge of $\Lagr{f}$ is equivalent to that of $\Lambda_f^+$ outside of $T^*_{=0}(M \times \R)$, and in particular it is equivalent for Lipschitz functions.
Note that Item~2 was proved in essence by N.~Vichery (\cite{Vic}*{Lemma~4.12}).

\begin{proof}
(1)
follows from the triangular inequality applied to the exact sequence~\eqref{eq:exact1}.

\spa
(2)
Consider the topological automorphism of $M \times \R$ given by $\Phi(x, t) \coloneqq (x, t - f(x))$.
Then, $\Phi(\Gamma^\pm_f) = M \times (\pm \R_{\geq 0})$.
Therefore, since $\closure{\interiorSmall{\Gamma^+_f}} = \Gamma^+_f$, we have $\muSupp(\Gamma^+_f) = \muSupp(\interiorSmall{\Gamma^+_f})^a$ by Equation~\eqref{eq:SSanti} and the fact that it is a topological and local property.
On the other hand, the triangular inequality applied to the exact sequence~\eqref{eq:exact2} implies that $\unpLagr{f}^- = \unpmuSupp(\interiorSmall{\Gamma^+_f})$.

\spa
(3)
The result is clear if $f$ is $C^1$.
Now, let $f$ be continuous.
Since the result to prove is local, we can assume that $M$ is compact.
We define an increasing sequence $(f_n)_{n \in \N}$ of smooth real-valued functions on $M$ converging pointwise (indeed, uniformly) to $f$ as follows.
Let $f_0$ be the constant function equal everywhere to $(\min_M f) - 1$, and given $f_n < f$, let $\epsilon_n \coloneqq \min_M (f-f_n) > 0$ and by density of smooth functions in the space of continuous functions with the compact-open topology, let $f_{n+1}$ be a smooth function such that $f - \epsilon_n / 2 < f_{n+1} < f$.

Then, $\Gamma^+_f \subseteq \Gamma^+_{f_{n+1}} \subseteq \Gamma^+_{f_n}$ and $\bigcap_{n \in \N} \Gamma^+_{f_n} = \Gamma^+_f$.
It follows that the inclusions induce an isomorphism $\indlim[n] \cor_{\Gamma^+_{f_n}} \isoto \cor_{\Gamma^+_f}$.
Applying~\cite{KS}*{Exe.~V.7}, one gets $\Lambda^+_f \subseteq \closure{\bigcup_n \Lambda^+_{f_n}} \subseteq T^*_{\geq 0}(M \times \R)$.
\end{proof}

\begin{example}
It can happen that $(\Lambda^+_f)_x = T^*_xM \times \R_{\geq 0}$.
An example is given by $f \colon \R \to \R, x \mapsto x \sin(1/x)$, for which $(\Lambda^+_f)_0 = \R \times \R_{\geq 0}$.
In particular, the union in Item~1 need not be disjoint (although it is disjoint outside of $T^*_{=0}(M \times \R)$ in view of Items~2 and~3, and in particular for Lipschitz functions).
\end{example}

\section{Main results}
\label{sec:general}

This is the main section of the paper, where we prove the characterizations of Lipschitz continuity and strict differentiability in terms of the conormal.
We also prove upper bounds on the Whitney cone and on the conormal of a continuous map.

\subsection{Mean value theorem}

In the case of a continuous map between vector spaces, we can give a lower bound on the conormal in the form of a mean value theorem.
As in the classical treatment, we prove it first for maps of a real variable.
Recall that the conormal $\Lagr{f}$ of a continuous map $f$ was defined in Equation~\eqref{eq:conormal}.

\begin{lemma}[Rolle's lemma]
Let $a, b \in \R$ with $a < b$.
If $f \colon [a, b] \to \R$ is a continuous function with $f(a) = f(b)$, then there exists $c \in \intervO{a, b}$ such that $\{0\} \times \R \subseteq \Lagr[c]{f}$.
\end{lemma}

\begin{proof}
Apply Fermat's lemma (Proposition~\ref{prop:fermat}) at a local extremum of $f$ in $\intervO{a, b}$.
\end{proof}

\begin{lemma}[mean value theorem for real-valued functions of a real variable]\label{lem:meanRR}
Let $a, b \in \R$ with $a < b$.
If $f \colon [a, b] \to \R$ is a continuous function, then there exists $c \in \intervO{a, b}$ such that $(f(b) - f(a), a - b) \in \Lagr[c]{f}$.
\end{lemma}

\begin{proof}
As in the classical case, we apply Rolle's lemma to the function $x \mapsto f(x) - \frac{f(b) - f(a)}{b - a} x$ and we use the shearing lemma (Lemma~\ref{lem:coneShear}).
\end{proof}

The previous Rolle lemma is not true for real-valued functions of several variables.
For instance, consider $p_2 \colon \R^2 \to \R$.
Then, $p_2(0, 0) = p_2(1, 0) = 0$, but for all $t \in [0, 1]$, one has $\Lagr[(t, 0)]{p_2} = \R (-1, 1)$, which intersects $\{0\} \times \R$ only at $\{(0, 0)\}$.
However, if we relax slightly the conclusion, there is a mean value theorem for continuous maps, which we prove first in the case of real-valued functions.
Recall that we defined $\unp{A} \coloneqq A \setminus \{0\}$.

\begin{lemma}[mean value theorem for real-valued functions]\label{lem:meanRealFunc}
Let $U$ be an open subset of a vector space.
Let $f \colon U \to \R$ be a continuous function.
Let $a, b \in U$ be such that $a \neq b$ and $[a, b] \subseteq U$.
Then, there exists $c \in \intervO{a, b}$ such that $\unpLagr[c]{f} \cap \big( b - a, f(b) - f(a) \big)^\perp \neq \varnothing$.
\end{lemma}

\begin{proof}
Set $u \coloneqq b - a$ and $v \coloneqq f(b) - f(a)$.
Define $\gamma \colon \intervO{0, 1} \to U, t \mapsto (1-t)a + tb$.
One has $\Lagr[t]{\gamma} = {\Gamma_\gamma}^\perp = (1, u)^\perp$ for any $t \in \intervO{0, 1}$.

If the pair $(\gamma, f)$ is not regular, then there exists a nonzero vector $\xi \in (0^*_\R \circ \Lagr{\gamma}) \cap (\Lagr{f} \circ 0^*_\R)$ with $\xi \in T^*_cU$ for some $c = \gamma(s) \in \intervO{a, b}$.
Therefore, $(0, \xi) \in \Lagr[s]{\gamma} = (1, u)^\perp$, so $\innerp{\xi}{u} = 0$, so $(\xi, 0) \in \unpLagr[c]{f} \cap (u, v)^\perp$, which completes the proof.

If the pair $(\gamma, f)$ is regular, then the chain rule for conormals (Proposition~\ref{prop:chain}) applies, and one has $\Lagr{f \circ \gamma} \subseteq \Lagr{\gamma} \acomp \Lagr{f}$.
We apply the mean value theorem for real-valued functions of a real variable (Lemma~\ref{lem:meanRR}) to $f \circ \gamma$ extended by continuity to $[0, 1]$.
It gives an $s \in \intervO{0, 1}$ such that $(v, -1) \in \Lagr[s]{f \circ \gamma}$.
Let $c \coloneqq \gamma(s)$.
By the chain rule, there exists $\xi \in T^*_c U$ such that $(v, \xi) \in \Lagr[s]{\gamma} = (1, u)^\perp$ and $(\xi, 1) \in \Lagr[c]{f}$.
One has $\innerp{(\xi, 1)}{(u, v)} = \innerp{\xi}{u} + v = 0$.
\end{proof}

We can now prove a mean value theorem for continuous maps between vector spaces.
Recall that the notation $\unpLagr[c][\eta_0]{f}$ was defined in Equation~\eqref{eq:directionalConormal}.

\begin{theorem}[mean value theorem]\label{thm:meanValThm}
Let $\V$ and $\W$ be vector spaces and $U \subseteq \V$ be open.
Let $f \colon U \to \W$ be a continuous map.
Let $a, b \in U$ be such that $a \neq b$ and $[a, b] \subseteq U$.
Let $\eta_0 \in \unp{\W}^*$.
Then, there exists $c \in \intervO{a, b}$ such that
\begin{equation}
\unpLagr[c][\eta_0]{f} \cap \big( b - a, f(b) - f(a) \big)^\perp \neq \varnothing.
\end{equation}
In particular, if $f$ is Lipschitz at $c$ for $(b-a)^\perp$, then there exists $\xi_0 \in \V^*$ such that $(\xi_0, \eta_0) \in \Lagr[c]{f} \cap \big( b - a, f(b) - f(a) \big)^\perp$.
\end{theorem}

\begin{proof}
Set $u \coloneqq b - a$ and $v \coloneqq f(b) - f(a)$ and $v_0 \coloneqq \innerp{\eta_0}{v}$.
Since $\eta_0 \colon \W \to \R$ is Lipschitz for $T^* \R$, the chain rule gives $\Lagr{\eta_0 \circ f} \subseteq \Lagr{f} \acomp \Lagr{\eta_0}$.
One has $\Lagr[x]{\eta_0} = {\Gamma_{\eta_0}}^\perp = \R (-\eta_0, 1)$ for all $x \in \intervO{a, b}$.
Applying Lemma~\ref{lem:meanRealFunc} to this real-valued function, we obtain $c \in \intervO{a, b}$ and a nonzero vector $(\xi, \tau) \in \unpLagr[c]{\eta_0 \circ f} \cap (u, v_0)^\perp$.

By the chain rule, there exists $\eta \in \W^*$ such that $(\xi, \eta) \in \Lagr[c]{f}$ and $(-\eta, \tau) \in \Lagr[f(c)]{\eta_0} = \R (-\eta_0, 1)$.
This implies $\eta = \tau \eta_0$.
Therefore, $(\xi, \tau \eta_0) \in \Lagr[c][\eta_0]{f}$, and it is nonzero since $(\xi, \tau) \neq 0$.
Moreover, $\innerp{(\xi, \tau \eta_0)}{(u, v)} = \innerp{(\xi, \tau)}{(u, v_0)} = 0$.

Finally, if $f$ is Lipschitz at $c$ for $(b-a)^\perp$, then $\tau \neq 0$, and since the conormal of $f$ is symmetric, we can suppose by $\R$-homogeneity that $\eta = \eta_0$.
\end{proof}

\subsection{Lower bound on the conormal}

We immediately obtain from the mean value theorem an upper bound on the Whitney cone of a continuous map in terms of its conormal.
We will also consider it as a sort of lower bound on its conormal in terms of its Whitney cone.
Recall that we defined $A^\top \coloneqq \bigcup_{v \in \unp{A}} v^\perp$.

\begin{theorem}[lower bound on the conormal]\label{thm:lower}
Let $f \colon M \to N$ be a continuous map between manifolds and let $x \in M$.
Then,
\begin{equation}
\Whit[x]{f} \subseteq \bigcap_{\eta \in \unp{T}^*_{f(x)}N} \Lagr[x][\eta]{f}^\top
\end{equation}
with equality if $\dim_{f(x)} N = 1$, in which case it reads
\begin{equation}
\Whit[x]{f} = \Lagr[x]{f}^\top.
\end{equation}
In particular, if $w \in \Whit[x]{f}$ and $\eta \in T^*_{f(x)} N$ and $f$ is Lipschitz at~$x$ for $p_M(w)^\perp$, then there exists $\xi \in T^*_x M$ such that $(\xi, \eta) \in \Lagr[x]{f} \cap w^\perp$.
In particular, if $f$ is Lipschitz at~$x$ for $T^*_x M$, then $p_N(\Lagr[x]{f}) = T^*_{f(x)} N$.
\end{theorem}

\begin{proof}
(i)
Let $w \in \unpWhit[x]{f}$ and $\eta \in \unp{T}^*_{f(x)} N$ and fix charts at~$x$ and $f(x)$.
There exist sequences $y_n, z_n \to [n] x$ and $c_n > 0$ such that $c_n (z_n - y_n, f(z_n) - f(y_n)) \to[n] w$.
Since $w \neq 0$, we can suppose that $y_n \neq z_n$ for all $n \in \N$.
By the mean value theorem (Theorem~\ref{thm:meanValThm}), there exist sequences $x_n \in \intervO{y_n, z_n}$ and $\nu_n \in \unpLagr[x_n][\eta]{f} \cap (z_n - y_n, f(z_n) - f(y_n))^\perp$.
Since $x_n \to[n] x$, up to extracting a subsequence and normalizing $\nu_n$, we can suppose that $(\nu_n)$ converges.
Its limit, say $\nu$, is in $\unpLagr[x][\eta]{f} \cap w^\perp$.

\spa
(ii)
For the case of equality, suppose that $\dim_{f(x)}N = 1$, and let $w = (u, v) \in \Lagr[x]{f}^\top \setminus \{0\}$.

Suppose first that $f$ is Lipschitz.
There exists $(\xi, \eta) \in \unpLagr[x]{f}$ such that $\innerp{\xi}{u} + \eta v = 0$.
Since $f$ is Lipschitz, $u \neq 0$, so by the upper bound on the conormal (Theorem~\ref{thm:upper}\footnote{This does not create any circular argument.}), there exists $v' \in \R$ such that $(u, v') \in \unpWhit[x]{f}$ and $\innerp{\xi}{u} + \eta v' = 0$.
So $\eta (v - v') = 0$.
Since $f$ is Lipschitz, $\eta \neq 0$, so $v = v'$, so $w = (u, v') \in \Whit[x]{f}$.

In the general case, suppose that $w \notin \Whit[x]{f}$.
By Proposition~\ref{prop:submanWhit}(1), this implies that $\Gamma_f$ is the graph (in other coordinates) of a Lipschitz map.
The result then follows from the previous paragraph, since $\Whit[x]{f}$ and $\Lagr[x]{f}$ only depend on $f$ via its graph.

\spa
(iii)
For the last claim, let $w \in \Whit[x]{f}$ and $\eta \in T^*_{f(x)} N$.
The case $\eta = 0$ is trivial since $0 \in \Lagr[x]{f} \cap w^\perp$, so we suppose $\eta \neq 0$.
By~(i), there exist $\xi \in T^*_xM$ and $t \in \R_{\geq 0}$ such that $(\xi, t \eta) \in \unpLagr[x]{f} \cap w^\perp$.
Since $f$ is Lipschitz at~$x$ for $p_M(w)^\perp$, we have $t \neq 0$, so we can suppose $t = 1$, so $(\xi, \eta) \in \Lagr[x]{f} \cap w^\perp$.
\end{proof}

\begin{remark}
The condition in the theorem that $f$ be Lipschitz at~$x$ for $p_M(w)^\perp$ is necessary, as the function $\sqrt[3]{-} \colon \R \to \R$ shows.
\end{remark}

\subsection{Characterization of Lipschitz continuity}

In this subsection, we prove that a continuous map between manifolds $f \colon M \to N$ is Lipschitz if and only if it is ``Lipschitz for $T^*M$'' (Definition~\ref{def:muSubm}(3) or Equation~\eqref{eq:regular}).
The definitions and properties we use related to Lipschitz continuity are recalled in Appendix~\ref{app:strictDiff}.
We first need a technical lemma.

\begin{lemma}\label{lem:Banach}
Let $A$ be a topological space and $a_0 \in A$.
Let $n \in \N$ and $C, D \in \R_{> 0}$ with $C D < 1$.
Let $f \colon \R \times A \to \R^n$ be a continuous map with $f(0, a_0) = 0$ which is $C$-Lipschitz in its first variable in a neighborhood of $(0, a_0)$.
Let $\psi \colon A \times \R^n \to \R$ be a continuous map with $\psi(a_0, 0) = 0$ which is $D$-Lipschitz in its last $n$ variables in a neighborhood of $(a_0, 0)$.
Then, there exist $\eta > 0$, an open neighborhood~$U$ of~$a_0$, and a continuous map $\Psi \colon U \to \R$ such that for all $(t, a) \in \intervO{-\eta, \eta} \times U$, one has $t = \psi(a, f(t, a)) \;\Leftrightarrow\; t = \Psi(a)$.
\end{lemma}

\begin{proof}
It suffices to apply the Banach fixed-point theorem with continuous parameter to the map $g \colon \R \times A \to \R, (t, a) \mapsto \psi(a, f(t, a))$ in a suitable neighborhood of $(0, a_0)$ of the form $\intervC{-\eta, \eta} \times U$ with $\eta >0$ and $U$ an open neighborhood of $a_0$.
\end{proof}

Recall that we defined $\Whit[x][0]{f} \coloneqq \Whit[x]{f} \cap (\{0\} \times T_{f(x)}N)$ and $\Lagr[x][0]{f} \coloneqq \Lagr[x]{f} \cap (T^*_x M \times \{0\})$.

\begin{theorem}\label{thm:Lip}
Let $f \colon M \to N$ be a continuous map between manifolds and let $x \in M$.
The following are equivalent:
\begin{enumerate}
\item
$f$ is Lipschitz on a neighborhood of~$x$,
\item
$\Whit[x][0]{f} = \{0\}$,
\item
$\Lagr[x][0]{f} = \{0\}$.
\end{enumerate}
\end{theorem}

The implication (1)$\Rightarrow$(3) for real-valued maps was proved in essence in~\cite{Vic}*{Theorem~3.9(6)}.
We recall the proof here: if $f \colon M \to \R$ is, say, $C$-Lipschitz at $x_0 \in M$ in a given chart, then, setting $\gamma \coloneqq \{ (x, t) \in \R^{m+1} \mid t \geq C \norm{x} \}$, one has $\Gamma^+_f + \gamma \subseteq \Gamma^+_f$, so $N_{x_0}(\Gamma^+_f) \supseteq \interior{\gamma}$, so by Proposition~\ref{prop:SSbounds}, $\Lambda_{x_0}^+(f) \subseteq N_{x_0}(\Gamma^+_f)^\circ \subseteq \gamma^\circ$.
Therefore, $\Lambda_{x_0}^+(f) \circ \{0\} \subseteq \gamma^\circ \circ \{0\} = \{0\}$, and we conclude by Proposition~\ref{prop:epigraph}(1).

\begin{proof}

\spa
(2)$\Rightarrow$(1)
is (one implication of) Proposition~\ref{prop:LipCone}(1).

\spa
(3)$\Rightarrow$(2).
Let $(0, v) \in \Whit[x][0]{f}$.
Let $\eta \in \unp{T}_{f(x)}^* N$.
By the lower bound on the conormal (Theorem~\ref{thm:lower}), there exist $\xi \in T_x^* M$ and $t \in \R_{\geq 0}$ such that $(\xi, t \eta) \in \Lagr[x]{f}$ and $\innerp{(\xi, t \eta)}{(0, v)} = 0$.
Condition~(3) implies that $t \neq 0$.
Therefore, for all $\eta \in T_{f(x)}^* N$, one has $\innerp{\eta}{v} = 0$.
Therefore, $v = 0$.

\spa
(1)$\Rightarrow$(3).
Let $(x_0, \xi_0) \in \unp{T}^*M$.
We may suppose that $M$ is open in $\R^m$, that $N = \R^n$, and that $f$ is $C$-Lipschitz for some $C \in \R_{> 0}$.

\spa
(i)
We first prove that $(\xi_0, 0) \in \Ppg_{x_0}(f)$.
Let $\phi \colon M \times N \to \R$ be a function of class $C^1$ with $\phi(x_0, f(x_0)) = 0$ and $d\phi(x_0, f(x_0)) = (\xi_0, 0)$.
We will construct a basis of open neighborhoods $W$ of $(x_0, f(x_0))$ such that each topological embedding $p_M \circ i_W \colon \Gamma_f \cap \{ \phi < 0 \} \cap W \hookrightarrow p_M(W)$ induces an isomorphism in cohomology.
Then, Lemma~\ref{lem:criterFunc} will imply that $(\xi_0, 0) \in \Ppg_{x_0}(f)$.

Since $p_M(d\phi(x_0, f(x_0))) = \xi_0 \neq 0$, we may suppose that $\phi(x, y) = x_1 - \psi(x', y)$ in some neighborhood $W_0$ of $(x_0, f(x_0))$, where $x = (x_1, x') \in \R^m$.
Therefore, for any $W \subseteq W_0$, one has
\begin{equation*}
\Gamma_f \cap \{ \phi < 0 \} \cap W = \{ (x, f(x)) \in W \mid x_1 < \psi(x', f(x)) \}.
\end{equation*}
Since $\frac{\partial \psi}{\partial y}(x_0', f(x_0)) = - p_N(d\phi(x_0, f(x_0))) = 0$, there exists an open neighborhood of $(x_0', f(x_0))$ where $\psi$ is $1/(C+1)$-Lipschitz in its last $n$ variables.
Therefore, by Lemma~\ref{lem:Banach}, there exist $\eta > 0$, an open neighborhood $U'$ of $x_0'$, and a continuous map $\Psi \colon U' \to \R$ such that, setting $U_0 \coloneqq \intervO{(x_0)_1 - \eta, (x_0)_1 + \eta} \times U'$, we have for $x \in U_0$,
\begin{equation*}
x_1 < \psi(x', f(x))
\quad\Leftrightarrow\quad
x_1 < \Psi(x').
\end{equation*}
Then, for any open neighborhood $W \subseteq W_0 \cap (U_0 \times N)$ of $(x_0, f(x_0))$, one has $\Gamma_f \cap \{ \phi < 0 \} \cap W = \{ (x, f(x)) \in W \mid x_1 < \Psi(x') \}$.
Set $U \coloneqq p_M(W \cap \Gamma_f)$.
The projection $p_M \colon \{ (x, f(x)) \in W \mid x_1 < \Psi(x') \} \to \{ x \in U \mid x_1 < \Psi(x') \}$ is an isomorphism.
If $W$ is convex, then so is $U$, and the inclusion $i_U \colon \{ x \in U \mid x_1 < \Psi(x') \} \hookrightarrow U$ induces an isomorphism in cohomology.
Therefore, so does $p_M \circ i_W = i_U \circ p_M \colon \Gamma_f \cap \{ \phi < 0 \} \cap W \hookrightarrow U$.

Finally, we can find a basis of open convex neighborhoods $W$ of $(x_0, f(x_0))$ such that $p_M(W \cap \Gamma_f) = p_M(W)$.
For instance, we may set $U_n \coloneqq B (x_0, 1/n)$ and $W_n \coloneqq U_n \times B \big( f(U_n), 1/n \big)$ for $n \in \N_{> 0}$, where $B \big( f(U_n), 1/n \big)$ denotes the $1/n$-neighborhood of $f(U_n)$.

\spa
(ii)
We will prove that $B_\infty \big( (\xi_0, 0), \norm{\xi_0}/(C + 2) \big) \subseteq \Ppg_{x_0}(f)$, where the left hand side denotes the open ball in $\R^{m+n}$ centered at $(\xi_0, 0)$ with radius $\norm{\xi_0}/(C + 2)$ for the sup norm.

Let $(\xi, \eta) \in T^*_{(x_0, f(x_0))}(M \times N)$ with $\norm{\xi}, \norm{\eta} < \norm{\xi_0}/(C+2)$.
Define the linear automorphism $\Phi \coloneqq \id + (\xi, \eta) \otimes (e, 0)$ of $T_{(x_0, f(x_0))} (M \times N)$, where $e \in T_{x_0}M$ is such that $\norm{e} = \opb{\norm{\xi_0}}$ and $\innerp{\xi_0}{e} = 1$.

One has $\opb{\Phi} = \id - \frac{1}{1 + \innerp{\xi}{e}} (\xi, \eta) \otimes (e, 0)$.
Therefore, if $v \in T_{f(x_0)} N$, then $\opb{\Phi}(0, v) =  ( - \frac{\innerp{\eta}{v}}{1 + \innerp{\xi}{e}} e, v)$.
One has $\norm{\frac{\innerp{\eta}{v}}{1 + \innerp{\xi}{e}} e} \leq \frac{\norm{\eta}\norm{v}}{1 - \norm{\xi}\norm{e}} \norm{e} \leq \frac{1/(C+2)}{1 - (1/(C+2))} \norm{v} = \norm{v}/(C + 1)$.
Since $f$ is $C$-Lipschitz, this implies $\Phi(\Whit[x_0]{f}) \cap (\{0\} \times T_{f(x_0)} N) = \{0\}$.
By Proposition~\ref{prop:whitney}(8), one has $C_{(x_0, f(x_0))}(\Phi(\Gamma_f), \Phi(\Gamma_f)) = \Phi(\Whit[x_0]{f})$.
Therefore, by Proposition~\ref{prop:submanWhit}(1), $\Phi(\Gamma_f)$ is locally the graph of a Lipschitz map, say $g$.
The relation $\Phi(\Gamma_f) = \Gamma_g$ and Equation~\eqref{eq:opb-Ppg-linear} imply $\Phi^\intercal \big( \Ppg_{x_0}(g) \big) = \Ppg_{x_0}(f)$.
From~(i), one has $(\xi_0, 0) \in \Ppg_{x_0}(g)$.
Therefore, $(\xi_0 + \xi, \eta) = \Phi^\intercal(\xi_0, 0) \in \Ppg_{x_0}(f)$.

\spa
(iii)
Finally, there is a neighborhood $U$ of $x_0$ (contained in the fixed chart) such that $f$ is $C$-Lipschitz on $U$ and~(i) and~(ii) apply with $x_0$ replaced by any $x \in U$.
On the other hand, if $(x, y) \notin \Gamma_f$, then $\Ppg_{(x, y)}(\Gamma_f) = T^*_{(x, y)}(M \times N)$.
Therefore, there is a neighborhood $W$ of $(x_0, f(x_0))$ such that $W \times B_\infty \big( (\xi_0, 0), \norm{\xi_0}/(C + 2) \big) \subseteq \Ppg(\Gamma_f)$, which proves that $(\xi_0, 0) \notin \Lagr[x_0]{f}$.
\end{proof}

\subsection{Upper bound on the conormal}

In this subsection, we give an upper bound on the conormal of a map in terms of its Whitney cone.
Recall that the notation $\Whit[x][u]{f}$ was defined in Equation~\eqref{eq:directionalWhitney}.

\begin{theorem}[upper bound on the conormal]\label{thm:upper}
Let $f \colon M \to N$ be a continuous map between manifolds and let $x \in M$.
Then
\begin{equation}
\Lagr[x]{f} \subseteq \bigcap_{u \in \unp{T}_x M} \Whit[x][u]{f}^\top
\end{equation}
with equality if $\dim_x M = 1$, in which case it reads
\begin{equation}
\Lagr[x]{f} = \Whit[x]{f}^\top.
\end{equation}

In particular, if $u \in T_x M$ and $\nu \in \Lagr[x]{f}$ and $f$ is Lipschitz at~$x$ for $p_N(\nu)^\perp$, then there exists $v \in T_{f(x)} N$ such that $(u, v) \in \Whit[x]{f} \cap \nu^\perp$.
\end{theorem}

\begin{proof}
(i)
Let $\xi \in T^*_{(x, f(x))} (M \times N)$ and $u \in \unp{T}_x M$ be such that $\xi^\perp \cap \Whit[x][u]{f} = \{0\}$.
We have to prove that $\xi \notin \Lagr[x]{f}$.
Since $u \neq 0$, there is a linear automorphism $\Phi$ of $T_{(x, f(x))} (M \times N)$ such that $\opb{\Phi}(\{0\} \times T_{f(x)} N) \subseteq \xi^\perp \cap (\R u \times T_{f(x)} N)$.

Let $(0, v) \in C_{(x, f(x))}(\Phi(\Gamma_f), \Phi(\Gamma_f)) = \Phi(\Whit[x]{f})$.
Then, $\opb{\Phi}(0, v) \in \Whit[x]{f} \cap (\R u \times T_{f(x)} N) \cap \xi^\perp \subseteq \Whit[x][\pm u]{f} \cap \xi^\perp = \{0\}$, so $v = 0$.
Therefore, $\Phi(\Gamma_f)$ is locally the graph of a Lipschitz map, say $g$.
By Equation~\eqref{eq:opb-Ppg-linear}, one has $\Phi^\intercal(\Lagr[x]{g}) = \Lagr[x]{f}$.
 
Let $v \in T_{f(x)} N$.
One has $\innerp{\opb{{\Phi^\intercal}}(\xi)}{(0, v)} = \innerp{\xi}{\opb{\Phi}(0, v)} = 0$, so $p_N(\opb{{\Phi^\intercal}}(\xi)) = 0$, so by Theorem~\ref{thm:Lip}((1)$\Rightarrow$(3)), one has $\opb{{\Phi^\intercal}}(\xi) \notin \Lagr[x]{g}$.
Therefore, $\xi \notin \Phi^\intercal(\Lagr[x]{g}) = \Lagr[x]{f}$.

\spa
(ii)
The case of equality is a special case of Proposition~\ref{prop:upperBoundSubman}.\footnote{This does not create any circular argument.}

\spa
(iii)
For the last claim, let $u \in T_x M$ and $\nu \in \Lagr[x]{f}$.
The case $u = 0$ is trivial since $0 \in \Whit[x]{f} \cap \nu^\perp$, so we suppose $u \neq 0$.
By~(i) and~(ii), there exist $v \in T_{f(x)} M$ and $t \in \R_{\geq 0}$ such that $(t u, v) \in \unpWhit[x]{f} \cap \nu^\perp$.
Since $f$ is Lipschitz at~$x$ for $p_N(\nu)^\perp$, we have $t \neq 0$, so we can suppose $t = 1$, so $(u, v) \in \Whit[x]{f} \cap \nu^\perp$.
\end{proof}

\begin{remark}
In view of Proposition~\ref{prop:LipCone}(1), the implication (1)$\Rightarrow$(3) of Theorem~\ref{thm:Lip} is a special case of this upper bound on the conormal.
\end{remark}

\begin{example}
Let $f \colon \R^2 \to \R, (x_1, x_2) \mapsto x_1^2 \sin(1/x_1)$.
One has $\Whit[0]{f} = \{ u \in \R^3 \mid \abs{u_3} \leq \abs{u_1} \}$, so $\bigcap_{v \in \unp{T}_0M} \Whit[0][v]{f}^\top = \{ (\xi_1, 0, \xi_3) \in \R^3 \mid \abs{\xi_1} \leq \abs{\xi_3} \}$.
In this case, the upper bound is easily seen to be an equality.
\end{example}

\subsection{Characterization of strict differentiability}

The lower and upper bounds on the conormal allow us to derive the following characterization of strict differentiability.

\begin{proposition}\label{prop:strict-diff}
Let $f \colon M \to N$ be a continuous map between manifolds and let $x \in M$.
The following are equivalent:
\begin{enumerate}
\item
$f$ is strictly differentiable at~$x$,
\item
$\Whit[x][0]{f} = \{0\}$ and $\Whit[x]{f}$ is contained in a $(\dim_x M)$-dimensional vector subspace,
\item
$\Lagr[x][0]{f} = \{0\}$ and $\Lagr[x]{f}$ is contained in a $(\dim_{f(x)} N)$-dimensional vector subspace,
\end{enumerate}
and in that case, $\Whit[x]{f} = \Gamma_{T_x f}$ and $\Lagr[x]{f} = (\Gamma_{T_xf})^\perp$.
\end{proposition}

\begin{proof}
(2)$\Rightarrow$(1).
This is (one implication of) Proposition~\ref{prop:LipCone}(2), which also proves $\Whit[x]{f} = \Gamma_{T_x f}$.

\spa
(3)$\Rightarrow$(2).
Condition~(3) implies that $f$ is Lipschitz, so the lower bound theorem reads:
for any $(u, v) \in \Whit[x]{f}$ and any $\eta \in \unp{T}^*_{f(x)} N$, there exists $\xi \in T^*_x M$ such that $(\xi, \eta) \in \Lagr[x]{f}$ and $\innerp{\xi}{u} + \innerp{\eta}{v} = 0$.
Therefore, $\Lagr[x]{f}$ is an $n$-dimensional vector subspace and is the graph of a linear map $L \colon T^*_{f(x)} N \to T^*_x M$ such that $\innerp{L(\eta)}{u} + \innerp{\eta}{v} = 0$ for all $(u, v) \in \Whit[x]{f}$ and $\eta \in T^*_{f(x)} N$.
Therefore, $\innerp{\eta}{L^\intercal(u) + v} = 0$, so $v = - L^\intercal(u) $, so $\Whit[x]{f}$ is the graph of $-L^\intercal$, which is an $m$-dimensional vector subspace.

\spa
(1)$\Rightarrow$(3).
This is a consequence of the upper bound on the conormal, since when $f$ is strictly differentiable at $x$, one has $\bigcap_{u \in \unp{T}_x M} \Whit[x][u]{f}^\top = \bigcap_{u \in \unp{T}_x M} \{ (u, f'(x)u) \}^\perp = \left( \sum_{u \in \unp{T}_x M} \R \big( u, f'(x)u \big) \right)^\perp = (\Gamma_{T_x f})^\perp$.
\end{proof}

\subsection{Application to real-valued functions of a real variable}

In this subsection, we present easy applications of the bounds on the conormal to the case of real-valued functions of a real variable.
First, note that $(-)^\top$ is an involution on nonzero pointed symmetric cones in a two-dimensional space.
Let $I$ be an open interval of $\R$ and $f \colon I \to \R$ be a continuous function.
By the case of equality in the upper bound theorem, one has $\Lagr{f} = \Whit{f}^\top$, hence also $\Whit{f} = \Lagr{f}^\top$.

If $f$ is Lipschitz and $x \in I$, then $\Lambda^+_x(f) = N_x(\Gamma^+_f)^\circ$, as is easily deduced from Propositions~\ref{prop:Dini} and~\ref{prop:epigraph}(3), that is, the general upper bound on the microsupport of a closed subset (Proposition~\ref{prop:SSbounds}) is an equality for the epigraph of a Lipschitz function of a real variable.

We set $P \coloneqq \{ (x, y, u, v) \in T(I \times \R) \mid u v \geq 0 \}$ and $N \coloneqq \{ (x, y, \xi, \eta) \in T^*(I \times \R) \mid \xi \eta \leq 0 \}$.
Note that $N = P^\top$ and $\interior{N} = \interior{P}^\top \setminus 0^*_{I \times \R}$.

\begin{proposition}\label{prop:monotone}
Let $I$ be an open interval of $\R$ and $f \colon I \to \R$ be a continuous function.
\begin{enumerate}
\item
The following conditions are equivalent:
\begin{enumerate}
\item
$f$ is non-decreasing,
\item
$\Whit{f} \subseteq P$,
\item
$\Lagr{f} \subseteq N$.
\end{enumerate}
\item
The following conditions are equivalent:
\begin{enumerate}
\item
$f$ is injective with Lipschitz inverse,
\item
$f$ is a Whitney immersion,
\item
$f$ is a microlocal submersion.
\end{enumerate}
\item
The following conditions are equivalent:
\begin{enumerate}
\item
$f$ is a strictly increasing Lipschitz-embedding,
\item
$\unpWhit{f} \subseteq \interior{P}$,
\item
$\unpLagr{f} \subseteq \interior{N}$.
\end{enumerate}
\end{enumerate}
\end{proposition}

\begin{proof}
(1)
The equivalence (b)$\Leftrightarrow$(c) follows from the discussion preceding the proposition.
The implication (a)$\Rightarrow$(b) is straightforward.
As for the implication (c)$\Rightarrow$(a), let $a < b \in I$.
By Lemma~\ref{lem:meanRR}, there exists $c \in \intervO{a, b}$ such that $(f(a) - f(b), b - a) \in \Lagr[c]{f}$.
Therefore, the hypothesis $\unp{\Lambda}_f \subseteq \interior{N}$ implies $f(a) < f(b)$.
(One could prove (b)$\Rightarrow$(a) in a similar way, via an analogue of Lemma~\ref{lem:meanRR} for Whitney cones.)

\spa
(2)
The equivalence (b)$\Leftrightarrow$(c) follows from the discussion preceding the proposition.
Namely, one has $\Whit{f} \cap (TI \times 0_\R) \subseteq 0_{I \times \R}$ if and only if $\Lagr{f} \cap (0^*_I \times T\R) \subseteq 0^*_{I \times \R}$, and in that case, $f$ is injective and an open map by the Fermat lemma.
If $f$ is injective and open, then $\Whit{\opb{f}}$ is the image of $\Whit{f}$ by the ``flip'' $T(I \times \R) \to T(\R \times I)$, so $\Whit[][0]{\opb{f}} \subseteq 0_{\R \times I}$ is equivalent to $\Whit{f} \cap (TI \times 0_\R) \subseteq 0_{I \times \R}$, that is, $\opb{f}$ is Lipschitz if and only if $f$ is Whitney immersive.

\spa
(3) follows from~(1) and~(2) and Proposition~\ref{prop:LipCone}(1).
\end{proof}

\begin{remark}
The implication (2b)$\Rightarrow$(2a) follows directly from invariance of domain, but we gave the preceding proof for its elementary nature.
\end{remark}

\subsection{Application to causal manifolds}

In~\cite{JS}, the authors introduced the category of causal manifolds, in which the category of spacetimes (time-oriented connected Lorentzian manifolds) up to conformal isomorphisms embeds.
In~\cite{JS}*{Def.~1.7}, a \define{causal manifold} $(M, \gamma_M)$ was defined to be a connected manifold $M$ equipped with an open convex cone $\gamma_M \subseteq TM$ which is nowhere empty ($\gamma_x \neq \varnothing$ for all $x \in M$), and a {causal morphism} $f \colon (M, \gamma_M) \to (N, \gamma_N)$ was defined to be a morphism of manifolds such that $Tf(\closure{\gamma_M}) \subseteq \closure{\gamma_N}$.

In a vector bundle, we denote by $\fwclosure{-}$ the fiberwise closure.
Note that for a nowhere empty convex cone $\gamma$ in a vector bundle, one has $\gamma^{\circ \circ} = \fwclosure{\gamma}$.

\begin{remark}\label{remark:erratum}
The proof of~\cite{JS}*{Prop.~1.12} actually proves that if $(M, \gamma_M)$ and $(N, \gamma_N)$ are causal manifolds and $f \colon M \to N$ is a morphism of manifolds, then $\Lagr{f} \acomp \gamma_N^\circ \subseteq \gamma_M^\circ$ if and only if $Tf(\fwclosure{\gamma_M}) \subseteq \fwclosure{\gamma_N}$ (and not, as stated there, $Tf(\closure{\gamma_M}) \subseteq \closure{\gamma_N}$).
These conditions imply that $f$ is causal and are satisfied when $f$ is strictly causal or when $f$ is causal and $\fwclosure{\gamma_N} = \closure{\gamma_N}$.
Since time functions are $\R$-valued, this misstatement has no consequences on the rest of the paper, with the exception of~\cite{JS}*{Cor.~2.10}, in which the morphism $f$ should be assumed strictly causal.
\end{remark}

Here, we make the additional assumptions that the cone $\gamma_M$ of a causal manifold $(M, \gamma_M)$ is \define{proper}, in the sense that $(\closure{\gamma_M})_x$ does not contain any line for any $x \in M$, and is \define{continuous}, in the sense that $\closure{\gamma_M} = \fwclosure{\gamma_M}$.
One can check that continuity as defined here is equivalent to the continuity of the map $x \mapsto (\gamma_M)_x$ for any reasonable topology on the space of cones (as for instance defined in~\cite{FS} using the Hausdorff distance).
We set $\gamma_\R \coloneqq \R \times \R_{>0}$, so that $(\R, \gamma_\R)$ is a causal manifold.

We will extend some of the results of~\cite{JS} from smooth maps to continuous maps.

\begin{definition}
A \define{causal morphism} $f \colon (M, \gamma_M) \to (N, \gamma_N)$ is a continuous map such that $\closure{\gamma_M} \comp \Whit{f} \subseteq \closure{\gamma_N}$.
\end{definition}

\begin{proposition}\label{prop:causal-mor}
A causal morphism is Lipschitz.
\end{proposition}

\begin{proof}
Let $f \colon (M, \gamma_M) \to (N, \gamma_N)$ be a causal morphism and let $x \in M$.
One has $0 \in (\closure{\gamma_M})_x$.
If $(0, v) \in \Whit[x]{f}$, then $\R v \subseteq (\closure{\gamma_N})_x$, but that cone is proper, so $v = 0$.
\end{proof}

We have the following extension of~\cite{JS}*{Prop.~1.12}.

\begin{proposition}\label{prop:causal-morphism}
Let $(M, \gamma_M)$ and $(N, \gamma_N)$ be causal manifolds and let $f \colon M \to N$ be a continuous map.
Then, $f$ is a causal morphism if and only if $\Lagr{f} \acomp \gamma_N^\circ \subseteq \gamma_M^\circ$.
If $(M, \gamma_M) = (\R, \gamma_\R)$, then this is equivalent to $\Whit{f} \subseteq \R (\{1\} \times \closure{\gamma_N})$.
If $(N, \gamma_N) = (\R, \gamma_\R)$, then this is equivalent to $\Lagr{f} \subseteq \R (\gamma_M^\circ \times \{-1\})$.
\end{proposition}

\begin{proof}
(i)
The condition is necessary.
Let $x \in M$ and $\xi \in \Lagr[x]{f} \acomp (\gamma_N^\circ)_{f(x)}$.
Let $u \in (\gamma_M)_x$.
There exists $\eta \in (\gamma_N^\circ)_{f(x)}$ such that $(\xi, -\eta) \in \Lagr[x]{f}$.
By the upper bound on the conormal, one has $(\xi, -\eta) \in \Whit[x][u]{f}^\top$.
Therefore, since $f$ is Lipschitz, there exists $v \in T_{f(x)}N$ such that $(u, v) \in \unpWhit[x]{f}$ and $\innerp{(\xi, -\eta)}{(u, v)} = 0$.
Since $f$ is causal and $u \in \gamma_M$, one has $v \in \closure{\gamma_N} = \gamma_N^{\circ \circ}$.
Therefore, since $\eta \in \gamma_N^\circ$, one has $\innerp{\xi}{u} = \innerp{\eta}{v} \geq 0$.
This proves that $\xi \in \gamma_M^\circ$.

\spa
(ii)
The condition is sufficient.
First, note that the inclusion $\Lagr{f} \acomp \gamma_N^\circ \subseteq \gamma_M^\circ$ implies that $f$ is Lipschitz.
Indeed, for $x \in M$, one has $0 \in (\gamma_N^\circ)_{f(x)}$.
If $(\xi, 0) \in \Lagr[x]{f}$, then $\R \xi \subseteq (\gamma_M^\circ)_x$, but that cone is proper, so $\xi = 0$.

Let $x \in M$ and $v \in (\closure{\gamma_M})_x \comp \Whit[x]{f}$.
There exists $u \in (\closure{\gamma_M})_x$ such that $(u, v) \in \Whit[x]{f}$.
Let $\eta \in (\gamma_N^\circ)_{f(x)} \setminus \{0\}$.
By the lower bound on the conormal, one has $(u, v) \in \Lagr[x][-\eta]{f}^\top$.
Therefore, since $f$ is Lipschitz, there exists $\xi \in T^*_x M$ such that $(\xi, -\eta) \in \Lagr[x]{f}$ and $\innerp{(\xi, -\eta)}{(u, v)} = 0$.
By the hypothesis, one has $\xi \in (\gamma_M^\circ)_x$.
Therefore, since $u \in (\closure{\gamma_M})_x = \closure{(\gamma_M)_x}$, one has $\innerp{\eta}{v} = \innerp{\xi}{u} \geq 0$.
This proves that $v \in \gamma_N^{\circ\circ} = \closure{\gamma_N}$.

\spa
(iii)
The cases where the domain or the codomain is $(\R, \gamma_\R)$ are obvious.
\end{proof}

\begin{lemma}
Let $f \colon (M, \gamma_M) \to (N, \gamma_N)$ be a causal morphism and $A \subseteq T^*M$ be a closed cone.
If $A \cap \gamma_M^\circ \subseteq 0^*_M$ and $f$ is non-characteristic for $\gamma_N^\circ$, then $(A \acomp \Lagr{f}) \cap \gamma_N^\circ \subseteq 0^*_N$.
\end{lemma}

\begin{proof}
One has
\begin{align*}
\Lagr{f} \cap (A \times \gamma_N^{\circ a})
&\subseteq
\Lagr{f} \cap ((A \cap \gamma_M^\circ) \times \gamma_N^{\circ a}) \\
&\subseteq
\Lagr{f} \cap (0^*_M \times \gamma_N^{\circ a}) \\
&\subseteq
0^*_{MN}
\end{align*}
respectively because $f$ is a causal morphism (and by Proposition~\ref{prop:causal-morphism}), by hypothesis, and because $f$ is non-characteristic for $\gamma_N^\circ$.
This implies $(A \acomp \Lagr{f}) \cap \gamma_N^\circ \subseteq 0^*_N$.
\end{proof}

This lemma shows that the proof of~\cite{JS}*{Thm.~2.9} extends from the case of the causal morphism $f$ being $C^1$ to $f$ being merely continuous.
Namely,

\begin{theorem}[extending~\cite{JS}*{Thm.~2.9}]
Let $f \colon (M, \gamma_M) \to (N, \gamma_N)$ be a morphism of causal manifolds, let $\preceq$ be a closed causal preorder on $M$, and let $F \in \Derb (\cor_M)$.
Assume that
\begin{enumerate}
\item
$f$ is non-characteristic for $\gamma_N^\circ$,
\item
for any $x \in M$, the map $f$ is proper on $J^-(x)$,
\item
$\muSupp(F) \cap \gamma_M^\circ \subseteq 0^*_M$.
\end{enumerate}
Then,
\begin{equation}
\muSupp(\roim{f} F) \cap \interior{\gamma_N^\circ} = \varnothing.
\end{equation}
\end{theorem}

In~\cite{JS}, a time function on $(M, \gamma_M)$ was defined to be a smooth submersive causal morphism $\tau \colon (M, \gamma_M) \to (\R, \gamma_\R)$.
By the extension of~\cite{JS}*{Thm.~2.9} we just proved, we see that all subsequent results of~\cite{JS}, in particular Theorem~2.13 and Corollary~3.8, continue to hold if the definition of a time function is weakened as follows.

\begin{definition}
A \define{time function} on a causal manifold $(M, \gamma_M)$ is a microlocally submersive causal morphism $\tau \colon (M, \gamma_M) \to (\R, \gamma_\R)$.
A \define{Cauchy time function} is a time function which is proper on the cc-future and cc-past of any point.
\end{definition}

\begin{proposition}
Let $\tau \colon (M, \gamma_M) \to (\R, \gamma_\R)$ be a time function.
Then $\gamma_M \circ \Whit{\tau} \subseteq \gamma_\R$ and $\tau$~is an open map.
\end{proposition}

\begin{proof}
Let $(u, v) \in \Whit[x]{f}$ with $u \in \gamma_M$.
Let $\eta \in \unp{\gamma_N^\circ}$.
By the lower bound on the conormal, and since $\tau$ is causal hence Lipschitz, there exists $\xi \in T^*_x M$ such that $(\xi, -\eta) \in \Lagr[x]{f}$ and $\innerp{(\xi, -\eta)}{(u, v)} = 0$.
Since $f$ is causal, one has $\xi \in \gamma_M^\circ$ by Proposition~\ref{prop:causal-morphism}, and since $\tau$ is microlocally submersive, one has $\xi \neq 0$.
Therefore, $\innerp{\eta}{v} = \innerp{\xi}{u} > 0$.
Therefore, $v \in \gamma_N$.

A time function is an open map by Proposition~\ref{prop:open}.
\end{proof}

\section{Topological submanifolds}
\label{sec:submanifolds}

In this section, we extend some of the previous bounds and characterizations of Lipschitz continuity and strict differentiability to topological submanifolds.

\begin{proposition}\label{prop:lip}
Let $M$ be a $C^0$-submanifold of a manifold $P$ and let $x \in M$.
The following are equivalent:
\begin{enumerate}
\item
$M$ is locally at $x$ the graph of a Lipschitz map,
\item
$C_x(M, M)$ intersects trivially a $(\codim_x M)$-dimensional vector space $F \subseteq T_x P$,
\end{enumerate}
and in that case,
\begin{align}
\muSupp_x(M) &\subseteq \bigcap_{w \in T_xP \setminus F} \big( C_x(M, M) \cap (w + F) \big)^\top
\qquad\text{and}\\
C_x(M, M) &\subseteq \bigcap_{\nu \in T^*_x P \setminus F^\perp} \big( \muSupp_x(M) \cap (\nu + F^\perp) \big)^\top
\end{align}
and in particular, $\muSupp_x(M) \cap F^\perp = \{0\}$.

If $\codim_x M = 1$ and $C_x(M, M) \neq T_x P$, then both conditions are satisfied and
\begin{equation}
C_x(M, M) = \muSupp_x(M)^\top.
\end{equation}
\end{proposition}

\begin{proof}
The equivalence of the two conditions is given by Proposition~\ref{prop:submanWhit}(1).
The other claims are restatements of the upper and lower bounds on the conormal.
\end{proof}

We have the following strengthening of Proposition~\ref{prop:subman}.
A $C^0$-submanifold is said to be \define{strictly differentiable} at a point if it is locally at that point the graph of a map which is strictly differentiable at that point.

\begin{proposition}\label{prop:strict-diff-subman}
Let $M$ be a $C^0$-submanifold of a manifold $P$ and let $x \in M$.
The following are equivalent:
\begin{enumerate}
\item
$M$ is strictly differentiable at $x$,
\item
$C_x(M, M)$ is included in a $(\dim_x M)$-dimensional vector subspace of $T_x P$,
\end{enumerate}
and in that case, the inclusion in Item~2 is an equality and $\muSupp_x(M) = C_x(M, M)^\perp$.
\end{proposition}

\begin{proof}
The equivalence of the two conditions is Proposition~\ref{prop:submanWhit}(2).
If these conditions are fulfilled, then there exist an open neighborhood $U$ of $x$ and a chart $\phi = (\phi_M, \phi_N) \colon U \isoto U_M \times U_N \subseteq \R^m \times \R^n$ and a map $f \colon U_M \to U_N$ strictly differentiable at $\phi_M(x)$ such that $\phi(M \cap U) = \Gamma_f$.
We write $T_x f$ for $T_{\phi_M(x)} f$ and similarly for the Whitney cone and the conormal of $f$.
The equality $\muSupp_x(M) = C_x(M, M)^\perp$ follows from $\Lagr[x]{f} = \Whit[x]{f}^\perp$ (by Proposition~\ref{prop:strict-diff}), $T_x \phi (C_x(M, M)) = \Gamma_{T_x f} = \Whit[x]{f}$ (Propositions~\ref{prop:submanWhit}(2) and~\ref{prop:strict-diff}), and $\muSupp_x(M) = \phi'(x)^\intercal (\Lagr[x]{f})$ (consequence of $\phi(M \cap U) = \Gamma_f$ and Equation~\eqref{eq:opbSet}).
\end{proof}

The following lemma is an analogue of Rolle's lemma (or the mean value theorem) for one-dimensional $C^0$-submanifolds.

\begin{lemma}\label{lem:meanValThmSubman}
Let $\V$ be a vector space and $f \colon [0, 1] \to \V$ be a continuous injection.
Set $M \coloneqq f \big( \intervO{0, 1} \big)$.
If $\eta \in \big( f(1) - f(0) \big)^\perp$, then there exists $t \in \intervO{0, 1}$ such that $\eta \in \muSupp_{f(t)}(M)$.
\end{lemma}

\begin{proof}
Since $\eta \in (f(1) - f(0))^\perp$, the continuous function $\eta \circ f$ has an extremum at some $t \in \intervO{0, 1}$, say a maximum.
We set $x \coloneqq f(t)$.
One has $M \subseteq \{ v \in \V \mid \innerp{\eta}{v - x} \leq 0 \}$, so $-\eta \in C_x(M)^\circ$.
By Proposition~\ref{prop:SSbounds}, this implies $-\eta \in \muSupp_x(M)$, which is enough by Proposition~\ref{prop:symmetric} (since $f$ is injective and $\intervC{0, 1}$ is compact, $M$ is an embedded submanifold, and it is locally flat since it is 1-dimensional).
\end{proof}

\begin{proposition}\label{prop:upperBoundSubman}
Let $M$ be a closed $C^0$-submanifold of a manifold $P$.
Then,
\begin{equation}
\muSupp(M) \subseteq C(M, M)^\top
\end{equation}
with equality if $\dim M = 1$.
\end{proposition}

Recall that on the other hand, one has $\opb{\pi_P}(M) \cap \closure{C(M)^\circ} \subseteq \muSupp(M)$.

\begin{proof}
Let $(x, \eta) \in \unpmuSupp(M)$.
Then, $\eta \notin C_x(M, M)^\top$ is equivalent to $C_x(M, M) \cap \eta^\perp = \{0\}$.
If this is the case, then by Proposition~\ref{prop:lip}, the submanifold $M$ is locally at~$x$ the graph of a map $f$.
Then, by the upper bound on the conormal, we obtain $\muSupp_x(M) = \Lagr[x]{f} \subseteq \Whit[x]{f}^\top = C_x(M, M)^\top$.

In the 1-dimensional case, equality follows from Lemma~\ref{lem:meanValThmSubman}.
Indeed, let $u \in \unp{C}_x(M, M)$ and $\xi \in u^\perp$.
Let $f \colon \intervO{-1, 1} \to M$ be a parametrization of $M$ in a neighborhood of $x$.
Then, there exist sequences or reals $x_n, y_n \in \intervO{-1, 1}$ and $c_n > 0$ with $x_n, y_n \to[n] 0$ and $c_n (f(y_n) - f(x_n)) \to[n] u$.
Since $u^\perp = -u^\perp$, we can suppose that $x_n < y_n$ for all $n \in \N$.
For each $n \in \N$, let $\xi_n$ be the vector of $(f(y_n) - f(x_n))^\perp$ closest to $\xi$.
Then $\xi_n \to[n] \xi$.
Applying Lemma~\ref{lem:meanValThmSubman} in each interval $\intervC{x_n, y_n}$ gives $z_n \in \intervO{x_n, y_n}$ such that $\xi_n \in \muSupp_{f(z_n)}(M)$.
One has $z_n \to[n] 0$, so $f(z_n) \to[n] x$, so $\xi \in \muSupp_x(M)$.
\end{proof}

\appendix

\section{Appendix: Lipschitz continuity and strict differentiability}
\label{app:strictDiff}

In this appendix, we recall standard definitions related to Lipschitz continuity and differentiability, mainly to set the notation.

\begin{definition}[(pointwise) $C$-Lipschitz function]
\label{def:lip}
Let $f \colon X \to Y$ be a function between metric spaces and let $C \in \R$.
\begin{enumerate}
\item
The function $f$ is \define{$C$-Lipschitz} if
\begin{equation}
d_Y(f(x_1), f(x_2)) \leq C \, d_X(x_1, x_2)
\end{equation}
for any $x_1, x_2 \in X$.
\item
Let $x_0 \in X$.
The function $f$ is \define{$C$-Lipschitz} (resp.\ \define{pointwise $C$-Lipschitz}) \define{at $x_0$} if for all $\epsilon > 0$ there exists a neighborhood $U$ of $x_0$ such that $f$ is $(C+\epsilon)$-Lipschitz on $U$ (resp.\ such that $d_Y(f(x_0), f(x)) \leq (C+\epsilon) \, d_X(x_0, x)$ for any $x \in U$).
\item
The infimum of the numbers $C$ such that $f$ is $C$-Lipschitz (resp.\ pointwise $C$-Lipschitz) at $x_0$ is called the \define{Lipschitz constant} (resp.\ \define{pointwise Lipschitz constant}) of $f$ at $x_0$ and is denoted by $\Lip_{x_0}(f)$ (resp.\ $\Lippw_{x_0}(f)$).
\end{enumerate}
\end{definition}

It is easy to prove that
\begin{equation}
\Lip (f) = \limsup (\Lippw(f)) \colon X \to \closure{\R}.
\end{equation}
The extended reals $\Lip_{x_0}(f)$ and $\Lippw_{x_0}(f)$ in Definition~\ref{def:lip}(3) are actually minima when finite.

For maps between manifolds, one can read Lipschitz properties in charts, and because of rescaling, the notion of $C$-Lipschitz continuity makes no sense anymore.
One can only require a map to be $C$-Lipschitz for \emph{some} $C > 0$, or for \emph{all} $C > 0$, in the following sense.

\begin{definition}[(pointwise) Lipschitz and (strictly) differentiable maps]
Let $f \colon M \to N$ be a continuous map between manifolds and let $x_0 \in M$.
\begin{enumerate}
\item
The map $f$ is \define{Lipschitz} (resp.\ \define{pointwise Lipschitz}) at $x_0$ if there exist charts $U$ at $x_0$ and $V$ at $f(x_0)$ and a constant $C \in \R$ such that in these charts $f$ is $C$-Lipschitz (resp.\ $C$-pointwise Lipschitz) at $x_0$.
It is \define{Lipschitz} if it is Lipschitz at~$x$ for all $x \in M$.
\item
The map $f$ is \define{strictly differentiable} (resp.\ \define{differentiable}) at $x_0$ if there exist a linear map $L \colon T_{x_0}M \to T_{f(x_0)}N$ and charts $U$ at $x_0$ and $V$ at $f(x_0)$ such that in these charts, for all $\epsilon > 0$, the map $f - L$ is $\epsilon$-Lipschitz (resp.\ pointwise $\epsilon$-Lipschitz) at $x_0$, that is, there exists a neighborhood $U_\epsilon \subseteq U$ of $x$ such that $f-L$ is $\epsilon$-Lipschitz on $U_\epsilon$.
\end{enumerate}
\end{definition}

The second item of this definition is of course a rewording of the usual definitions (if $L$ exists, it is unique and equal to $T_{x_0}f$).
It emphasizes the naturality of the notion of strict differentiability.
Strict differentiability is the good notion of ``$C^1$ at a point'': if a map is differentiable on a neighborhood of a point, then it is strictly differentiable at that point if and only if its derivative is continuous at that point.
It is also the natural hypothesis for the inverse function theorem.
Strict differentiability at a point implies Lipschitz continuity in a neighborhood of that point.

The following lemma is used to prove Lemma~\ref{lem:SSstrictDiff}.

\begin{lemma}\label{lem:strict-diff}
Let $U$ be an open subset of $\R^n$ and $x \in U$.
Let $\gamma \subseteq \R^n$ be a closed cone.
Let $\phi \colon U \to \R$ be a function which is strictly differentiable at $x$ with $\phi(x) = 0$ and $d\phi(x) \in \interior{\gamma^{\circ a}}$.
Then, there exists an open neighborhood $V \subseteq U$ of $x$ such that $V \cap ((V \cap \{ \phi < 0 \}) + \gamma) \subseteq V \cap \{ \phi < 0 \}$, that is, $V \cap \{ \phi < 0 \}$ is $\gamma$-open in $V$ in the sense of~\cite{KS}*{Def.~3.2.1}.
\end{lemma}

\begin{proof}
Set $L \coloneqq d\phi(x)$ and $\epsilon \coloneqq \min \{ \abs{\innerp{L}{u}} \mid u \in \gamma \text{ and } \norm{u} = 1 \}$.
Let $V \subseteq U$ be an open neighborhood of $x$ where the function $\phi - L$ is $\epsilon$-Lipschitz.
Let $y \in V \cap \{ \phi < 0 \}$ and $u \in \gamma$ be such that $y + u \in V$.
Then, $\phi(y + u) \leq \phi(y) + \innerp{L}{u} + \epsilon \norm{u} < 0$.
\end{proof}

\section{Appendix: tangent cones}
\label{app:cones}

In this appendix, we recall the definitions of the tangent cone $C(A)$, the strict tangent cone $N(A)$, and the Whitney cone $C(A, B)$ of subsets $A, B$ of a manifold, and their main properties.
The (elementary) proofs can be found in~\cite{JS}*{App.~A}.

Let $M$ be a manifold.
Let $x \in M$, let $x_n, y_n \to[n] x$ and $(c_n) \in \R^\N$ be three sequences, and let $u \in T_xM$.
We write ``$c_n (y_n - x_n) \to[n] u$'' to mean that in some chart $\phi$ at~$x$, one has $c_n (\phi(y_n) - \phi(x_n)) \to[n] T_x\phi(u)$.
This then holds for any chart at~$x$.
If $A, B \subseteq M$, their \define{Whitney cone} (\cite{KS}*{Def.~4.1.1 and Prop.~4.1.2}) is defined as
\begin{multline}
C_x(A, B) \coloneqq \{ u \in T_x M \mid \exists (x_n) \in A^\N, \exists (y_n) \in B^\N, \exists (c_n) \in (\R_{> 0})^\N \text{ such that}\\
x_n \to[n] x \text{ and } y_n \to[n] x \text{ and } c_n (y_n - x_n) \to[n] u \}.
\end{multline}
We write $C(A, B) \coloneqq \bigcup_{x \in M} C_x(A, B)$ (and similarly for the cones defined below).
We define the \define{tangent cone} of $A$ as
\begin{equation}
C_x(A) \coloneqq C_x(\{x\}, A).
\end{equation}
We define the \define{strict tangent cone} (or strict normal cone, see~\cite{KS}*{Def.~5.3.6}) of $A$ as
\begin{equation}
N(A) \coloneqq TM \setminus C(A, M \setminus A).
\end{equation}

One has
\begin{equation}
\opb{\tau_M}(\closure{A}) \cap N(A) \subseteq C(A) \subseteq C(A, A).
\end{equation}

\begin{proposition}[elementary properties of the Whitney cone]\label{prop:whitney}
Let $M$ be a manifold and $A, B, A_1, A_2 \subseteq M$.
\begin{enumerate}
\item
The Whitney cone $C(A, B) \subseteq TM$ is a closed cone.
\item
antisymmetry:
$C(A, B) = - C(B, A)$.
\item
monotony:
if $A_1 \subseteq A_2$, then $C(A_1, B) \subseteq C(A_2, B)$.
\item
additivity:
$C(A_1 \cup A_2, B) = C(A_1, B) \cup C(A_2, B)$.
\item
stability under closure:
$C(\closure{A}, B) = C(A, B)$.
\item
projection on the manifold:
$\tau_M(C(A, B)) = \closure{A} \cap \closure{B}$.
\item
If $x \in \interior{A} \cap \closure{B}$, then $C_x(A, B) = T_xM$.
\item
If $f \colon M \to N$ is strictly differentiable at $x \in M$, then $T_xf(C_x(A, B)) \subseteq C_{f(x)}(f(A), f(B))$.
\end{enumerate}
\end{proposition}

These properties imply corresponding properties for the tangent cone $C(A)$ (and in that case, strict differentiability in Item~8 can be replaced by differentiability).
The tangent cone $C(A)$ is pointwise closed but need not be closed, and one can have $C(A) \subsetneq \closure{C(A)} \subsetneq C(A, A)$.
Note also that $C(A, A)$ is a closed symmetric cone but need not be convex.

\begin{example}
Let $A$ be the graph of the function $\abs{-} \colon \R \to \R$.
Then $C_0(A) = \closure{C(A)}_0 = A \subsetneq C_0(A, A) = \{ (u, v) \in \R^2 \mid \abs{v} \leq \abs{u} \}$.

Let $A$ be the graph of the function $\R \to \R, x \mapsto x^2 \sin (1/x)$.
Then $C_0(A) = \R \times \{0\} \subsetneq  \closure{C(A)}_0 = C_0(A, A) = \{ (u, v) \in \R^2 \mid \abs{v} \leq \abs{u} \}$.
\end{example}

The strict tangent cone $N(A)$ of a subset $A$ of a manifold $M$ is an open convex cone.
Let $V$ be a chart at $x \in M$.
One has $u \in N_x(A)$ if and only if there exist an open neighborhood $U \subseteq V$ of $x$ and an open conic neighborhood $\gamma \subseteq T_x U$ of $u$ such that in a chart, $U \cap ((U \cap A) + \gamma) \subseteq A$.

\begin{proposition}\label{prop:strict-cone}
Let $A$ be a subset of a manifold $M$.
One has $C(\partial A, \partial A) \cap N(A) = \varnothing$ and $ N_x(A) \cap N_x(M \setminus A) = \varnothing$ for any $x \in \partial A$.
\end{proposition}

\begin{proof}
(i)
$C_x(\partial A, \partial A) \cap N_x(A) = \varnothing$ for $x \in M$.
This is trivial if $x \notin \partial A$, so let $x \in \partial A$.
Let $u \in C_x(\partial A, \partial A)$.
There exist sequences $(x_n), (x'_n) \in (\partial A)^\N$ both converging to $x$ and a sequence $(c_n) \in (\R_{>0})^\N$ such that $c_n (x'_n - x_n) \to[n] u$.
If $u = 0$, the result is trivial, so we can suppose that $x_n \neq x'_n$ for any $n \in \N$.
For each $n \in \N$, there exist sequences $(x_{n, m})_m \in A^\N$ and $(x'_{n, m})_m \in (M \setminus A)^\N$ converging respectively to $x_n$ and $x'_n$.
Let $\rho \colon \N \to \N$ be a strictly increasing function such that $c_n (x_{n, \rho(n)} - x_n - x'_{n, \rho(n)} + x'_n) \to[n] 0$.
Then, the sequences $(x_{n, \rho(n)})$ and $(x'_{n, \rho(n)})$ and $(c_n)$ show that $u \in C_x(A, M \setminus A) = T_xM \setminus N_x(A)$.

\spa
(ii)
$N_x(A) \cap N_x(M \setminus A) = \varnothing$ for $x \in \partial A$.
Fix a chart at $x$ and suppose that there exists $u \in N_x(A) \cap N_x(M \setminus A)$.
There exist an open neighborhood $U$ of $x$ and an open conic neighborhood $\gamma \subseteq T_xM$ of $u$ such that $U \cap ( (U \cap A) + \gamma) \subseteq A$ and $U \cap ( (U \setminus A) + \gamma) \subseteq U \setminus A$.
Therefore, $U \cap ( (U \cap A) + \gamma) \cap ( (U \setminus A) + \gamma) = \varnothing$.
In particular, since the only conic neighborhood of 0 is $T_x M$, one has $u \neq 0$.
Let $(x_n) \in (U \cap A)^\N$ and $(y_n) \in (U \setminus A)^\N$ be sequences converging to $x$.
Rescaling $u$, we can suppose that $x + u \in U$.
The set $(x+u) - \gamma$ is a neighborhood of $x$, so for $n$ large enough, one has $x + u \in U \cap (x_n + \gamma) \cap (y_n + \gamma) \neq \varnothing$, a contradiction.
\end{proof}

\vspace*{5mm}
\noindent
\parbox[t]{21em}
{\scriptsize
Beno{\^i}t Jubin\\
Sorbonne Universit{\'e}s, UPMC Univ Paris 6\\
Institut de Math{\'e}matiques de Jussieu\\
F-75005 Paris France\\
e-mail: benoit.jubin@imj-prg.fr
}
\end{document}